\documentclass{article}
\usepackage {amsmath, amsfonts, amssymb, mathrsfs, amsthm,accents, sectsty,hyphenat, calc}
\usepackage[usenames]{color}
\usepackage[nohug,midshaft]{diagrams}
\usepackage{palatino}
\def\mf#1{\mathfrak{#1}}
\def\mc#1{\mathcal{#1}}
\def\mb#1{\mathbb{#1}}
\def\tx#1{{\rm #1}}
\def\tb#1{\textbf{#1}}

\def\ts#1{\textsf{#1}}
\def\tr{\tx{tr}\,}

\def\C{\mathbb{C}}
\def\Q{\mathbb{Q}}

\def\Z{\mathbb{Z}}

\def\ol#1{\overline{#1}}
\def\ul#1{\underline{#1}}

\def\hat{\widehat}

\def\rw{\rightarrow}
\def\lw{\leftarrow}

\def\lrw{\longrightarrow}

\def\thrw{\twoheadrightarrow}

\def\lw{\leftarrow}

\def\<{\langle}
\def\>{\rangle}

\def\pf{\tb{Proof: }}

\newenvironment{mytitle}
{\begin{center}\large\sc}
{\end{center}}

\newarrow{Equals}{=}{=}{=}{=}{}

\newtheorem{thm}{Theorem}[subsection]
\newtheorem{lem}[thm]{Lemma}
\newtheorem{pro}[thm]{Proposition}

\newtheorem{dfn}[thm]{Definition}
\newtheorem{fct}[thm]{Fact}

\sectionfont{\center\sc\normalsize}
\subsectionfont{\bf\normalsize}

\numberwithin{equation}{section}

\def\phi{\varphi}

\begin{document}

\begin{mytitle} Supercuspidal $L$-packets via isocrystals \end{mytitle}
\begin{center} Tasho Kaletha \end{center}
\begin{abstract} In a recent paper \cite{DR09}, DeBacker and Reeder construct and parameterize $L$-packets on pure inner forms
of unramified $p$-adic groups, that consist of depth zero supercuspidal representations. We generalize their work to non-pure inner forms, by providing an
alternative construction based on the theory of isocrystals with additional structure due to Kottwitz. Furthermore, we show the stability and endoscopic transfer for these $L$-packets.
\end{abstract}

Given a reductive group $G$ over a local field $F$, the local Langlands correspondence seeks to parameterize the irreducible admissible representations of
$G(F)$ by means of Langlands parameters. Each Langlands parameter is supposed to correspond to an $L$-packet -- a finite set of representations of $G(F)$ that
satisfies many properties. Recently, DeBacker and Reeder have considered groups $G$ which are pure inner forms of unramified groups, and a certain class of
elliptic Langlands parameters. To each such parameter they construct a finite set of depth zero supercuspidal representations of $G(F)$ and its pure inner forms. The stability of these packets was proved in
\cite{DR09}, and their endoscopic transfer was proved in \cite{Kal09}.

The purpose of this paper is to generalize this work to non-pure inner forms of unramified groups. The concept of pure inner forms was introduced by Vogan, who
realized that the notion of an inner form is not rigid enough for the purposes of the local Langlands correspondence. Unfortunately, not all inner forms can be
rigidified to pure inner forms. Kottwitz observed that isocrystals with additional structure can be used to provide a rigidification of inner forms of
$p$-adic groups that fits very naturally in the framework of the local Langlands correspondence and endoscopy, and moreover has the potential to rigidify all inner forms of a given quasi-split group. We use Kottwitz's idea and his work
\cite{Kot85}, \cite{Kot97} on isocrystals to give a construction of $L$-packets on these rigidified inner forms. Our construction is different from that of
DeBacker and Reeder, but gives rise to the same packets on pure inner forms. Moreover, we prove the stability and endoscopic transfer of our packets. In the case of a connected center, all inner forms can be rigidified using isocrystals. The case of a disconnected center can be reduced to that
of a connected center. The techniques and difficulties involved in this reduction are of a separate nature and will be pursued in a forthcoming paper.

To explain the results more precisely, let $F$ be a $p$-adic field with Weil group $W_F$ and $G$ an unramified connected reductive group defined over $F$. In
\cite{Kot97}, Kottwitz introduces a set $\tb{B}(G)_b$, which consists of certain cohomology classes, and shows how each element $b \in \tb{B}(G)_b$ gives rise
to an equivalence class of inner forms $G^b$ of $G$. Let $\hat G$ be a Langlands dual group for $G$, and $^LG = \hat G \rtimes W_F$ the corresponding
$L$-group. The first main result of this paper is the explicit construction and parametrization of a packet $\Pi_{[\phi]}^b$ of representations of $G^b(F)$ for each $b \in
\tb{B}(G)_b$ and each Langlands parameter
\[ \phi : W_F \rw {^LG} \]
which is subject to certain properties (see Section \ref{sec:parms}). This packet depends only on the $\hat G$-conjugacy class $[\phi]$ of $\phi$ and on no
auxiliary choices.

The packet $\Pi_{[\phi]}^b$ is parameterized in the following way: Consider the group $S_\phi = \tx{Cent}(\phi,\hat G)$. The properties of $\phi$ imply
that this is a diagonalizable subgroup of $\hat G$ and contains $Z(\hat G)^\Gamma$ as a subgroup of finite index. It depends only on $[\phi]$ (see the
beginning of Section \ref{sec:packs} for a discussion), and we write $S_{[\phi]}$ for it. For each hyperspecial vertex $o$ in the reduced building of $G(F)$,
we construct a bijection
\begin{equation} \label{eq:packs0} X^*(S_{[\phi]}) \rw \bigsqcup_{b \in \tb{B}(G)_b} \Pi_{[\phi]}^b.\end{equation}
We show that this bijection satisfies the following two properties: First, it maps the trivial character to a representation of $G(F)$ which is generic with
respect to a character of generic depth $0$ at $o$ (see \cite{DR08} for terminology). Second, it identifies each set $\Pi_{[\phi]}^b$ with the fiber over $b$
of the map
\[X^*(S_{[\phi]}) \thrw X^*(Z(\hat G)^\Gamma) \rw \tb{B}(G)_b,    \]
where the first map is the natural restriction and the second is the isomorphism constructed by Kottwitz.

We give an explicit formula for the dependence of the bijection \eqref{eq:packs0} on the choice of the hyperspecial vertex $o$. More precisely, for each pair
of hyperspecial vertices $o,o'$ we construct an element $(o,o') \in X^*(S_{[\phi]})$ and show that the two versions of \eqref{eq:packs0} corresponding to $o$
and $o'$ differ by translation by $(o,o')$.

Our second main result is the compatibility of \eqref{eq:packs0} with the construction of \cite{DR09}. The Langlands parameters we are
considering are the same as the ones considered there. In that paper, the authors construct for each pure inner form $G^u$ of $G$, where $u \in H^1(F,G)$, a
packet $\Pi_\phi^u$ on $G^u(F)$ and a bijection
\[ X^*(C_\phi) \rw \bigsqcup_{u \in H^1(F,G)} \Pi_\phi^u, \]
where $C_\phi = \pi_0(S_\phi)$. There is a natural injection $H^1(F,G) \rw \tb{B}(G)_b$ and if $b$ is the image of $u$, then $G^b=G^u$. We show that in this
case, we have $\Pi_{[\phi]}^b=\Pi_\phi^u$, and moreover we have a commutative diagram
\begin{diagram}
X^*(S_{[\phi]})&&\rTo&\bigsqcup_{b \in \tb{B}(G)_b} \Pi_{[\phi]}^b\\
\uInto&&&\uInto\\
X^*(C_{\phi})&&\rTo&\bigsqcup_{u \in H^1(F,G)} \Pi_{\phi}^u\\
\end{diagram}
In this sense, our work is an extension of theirs. While our construction uses different methods and in particular does not rely on the combinatorial arguments
in \cite{DR09} involving the Bruhat-Tits building, the reader will clearly notice the influence of this paper.

The remaining results concern character identities and rely on the character computations in \cite[\S\S8,9,10,12]{DR09}. These
computations are derived for general $p$-adic groups which split over an unramified extension of $F$, and can be used in our setting. For each $b \in \tb{B}(G)_b$ and $s \in S_{[\phi]}$ we define
\[ \Theta_{[\phi],b}^s = e(G^b)\sum_{\substack{\rho \in X^*(S_{[\phi]})\\\rho \mapsto b}} \rho(s) \Theta_{\pi_\rho}, \]
where $\rho \mapsto \pi_\rho$ is the map \eqref{eq:packs0}, $\Theta_{\pi_\rho}$ is the character of the representation $\pi_\rho$, which we view as a function
on the set of strongly regular semi-simple elements of $G^b(F)$, and $e(G^b)$ is the Kottwitz sign of $G^b$. Our third main result is the stability of the
family
\[ \left\{ \Theta_{[\phi],b}^1 \right\}_{b \in \tb{B}(G)_b} \]
in the following sense: There is a notion of when two semi-simple elements $\gamma^b \in G^b(F)$ and $\gamma^c \in G^c(F)$ for $b,c \in \tb{B}(G)_b$ are stably
conjugate (see Section \ref{sec:eptwists}), and we show that if this is the case, then
\[ \Theta_{[\phi],b}^1(\gamma^b) = \Theta_{[\phi],c}^1(\gamma^c). \]
Our final result is the validity of endoscopic induction for our $L$-packets. Let $(H,s,{^L\eta})$ be an unramified extended endoscopic triple for $G$,
$[\phi^H]$ a Langlands parameter for $H$, and $[\phi]=[^L\eta]\circ[\phi^H]$. Associated to the bijection \eqref{eq:packs0} there is a compatible family of
normalizations $\Delta^{G^b}_H$ of the transfer factors for all pairs $(G^b,H)$ with $b \in \tb{B}(G)_b$. We define the endoscopic lift to $G^b(F)$ of the
stable character corresponding to $[\phi^H]$ by
\[ \tx{Lift}^{G^b}_H\Theta_{[\phi^H],0}^1(\gamma^b) := \sum_{\gamma^H} \Delta^{G^b}_H(\gamma^H,\gamma^b)
\frac{D^H(\gamma^H)^2}{D^{G^b}(\gamma^b)^2}\Theta_{[\phi^H],0}^1(\gamma^H), \]%
where $\gamma^H$ runs over the stably conjugacy classes of strongly regular semi-simple elements of $H(F)$ and $D$ are the usual Weyl discriminants. We show
that
\[ \tx{Lift}^{G^b}_H\Theta_{[\phi^H],0}^1 = \Theta_{[\phi],b}^s. \]

We will now briefly describe the contents of this paper. In section \ref{sec:eptwists} we lay out Kottwitz's ideas on using isocrystals to rigidify inner forms
and obtain good notions of rational and stable conjugacy of elements lying in different inner forms, as well as compatible normalizations of transfer factors
across inner forms. We call the rigidified inner forms extended pure inner forms, or ep forms for short. The set $\tb{B}(G)_b$ parameterizes the equivalence
classes of ep forms. We want to be able to work with actual ep forms, rather than just equivalence classes, and we introduce for that purpose the set $E(G)$.
It is the analog of the set $Z^1(F,G)$ of pure inner forms. Section \ref{sec:eptwists} is valid for any connected reductive p-adic group. Section \ref{sec:clp} is devoted to the $L$-packets. In \ref{sec:reps} we review the
representations that constitute our packets. These representations are very well known and we limit ourselves to just gathering the properties that will be
important for us. For an exposition on their construction, we refer the reader to \cite[\S4.4]{DR09}. In \ref{sec:parms} we introduce the Langlands parameters for our packets. They are the same as the ones considered
in \cite{DR09}. The
actual construction of the packets and their parameterization is carried out in \ref{sec:packs}. In order to construct the packet for a given parameter, we
first construct a triple $(S_0,[a],[^Lj])$, which depends on the choice of a hyperspecial vertex $o$. During its construction, we make some auxiliary choices, but in \ref{sec:indep} we show that the properties of the triple
$(S_0,[a],[^Lj])$ make it essentially unique, so none of the auxiliary choices matters, thereby showing the canonicity of our construction. Moreover, we
give in \ref{sec:indep} the formula for the dependence of the bijection \eqref{eq:packs0} on the choice of the hyperspecial vertex $o$. In \ref{sec:oldpacks}
we prove the compatibility of our packets with those constructed by DeBacker and Reeder. Finally \ref{sec:generic} provides some facts needed for the proper
normalization of the transfer factors. Section \ref{sec:stabendo} deals with stability and endoscopic transfer. After
establishing a reduction formula for the unstable character in \ref{sec:schar}, we show stability and endoscopic transfer in \ref{sec:se}. In fact, the
endoscopic transfer follows from the arguments of \cite{Kal09}, after a few key statements have been generalized to our situation. Most notably, these
statements include the reduction formula from \ref{sec:schar}, and the sign computation from \ref{sec:seprep}. After that, the argument of \cite[\S7]{Kal09}
goes through verbatim, and we refer the reader to it. While the stability of our packets is a special case of their endoscopic transfer, the argument for
this special case is significantly simpler and shorter, and so we give a direct proof for it.

In order to derive our results, we have to impose some restrictions on the field $F$. For the construction of our packets in section \ref{sec:clp}, we
assume that $p$ is odd. In section \ref{sec:stabendo}, we impose the same restrictions on $F$ as in \cite[\S12]{DR09} and \cite[\S3]{Kal09}.
Moreover, we require in this section that the center of $G$ be connected. This last requirement is not strictly necessary, but considerably simplifies the proof of
Proposition \ref{pro:tf}. The validity of this proposition in general can be deduced via reduction to the case of a connected center, which will be done in the
forthcoming paper mentioned above.

We would like to mention that the constructions in Section \ref{sec:clp} are not limited to the depth-zero case. They apply equally well in situations of positive depth, and we can show that they provide a generalization to non-pure inner forms of the packets constructed in \cite{Ree08} as well. We will address this in a future paper, once the stability of these packets, which is a work in progress by DeBacker and Spice, has been established.

\tb{Acknowledgements:} The author is indebted to Robert Kottwitz for the generosity and kindness with which he has shared his knowledge and insights. The author also thanks Stephen DeBacker for the encouragement to work on this problem and for helpful conversations. Finally, the support of the National Science Foundation%
\footnote{This material is based upon work supported by the National Science Foundation under agreement No. DMS-0635607. Any opinions, findings and conclusions or recommendations expressed in this material are those of the author and do not necessarily reflect the views of the National Science Foundation.}
is gratefully acknowledged.

\newpage
\tableofcontents

\section{Notation and conventions}

Throughout this paper, $F$ denotes a $p$-adic field, i.e. a finite extension of $\Q_p$, with ring of integers $O_F$ and residue field $k_F$ of cardinality $q_F$. We fix an algebraic closure $\ol{F}$ of $F$, and let $\Gamma$ denote the absolute Galois group of $F$, $W_F$ the Weil group, and $I_F$ the inertia
subgroup. We fix a uniformizer $\pi \in O_F$, as well as an element $\Phi \in \Gamma$ whose inverse induces the map $x \mapsto x^{q_F}$ on $\ol{k_F}$. Let
$F^u$ be the maximal unramified extension of $F$ in $\ol{F}$, $L$ the completion of $F^u$, and $\ol{L}$ a fixed algebraic closure of $L$, which we take to contain
$\ol{F}$. Recall that $\ol{L} \cong L \otimes_{F^u} \ol{F}$. The diagonal action of $\Gamma$ on the right hand side of this isomorphism is well-defined and in
this way we obtain a continuous action of $\Gamma$ on $\ol{L}$ whose fixed field is $F$.

If $G$ is a connected reductive group defined over $F$, we will use the Fraktur letter $\mf{g}$ for its Lie-algebra. If $T \subset G$ is a maximal torus (by which we mean that it is
defined over $F$), we write $R(T,G)$ for the set of roots of $T$ in $G$, $N(T,G)$ for the normalizer and $\Omega(T,G)$ for the Weyl group of $T$ in $G$. If $B
\subset G$ is a Borel containing $T$ (again defined over $F$), then $R(T,B)$ is a choice of positive roots inside $R(T,G)$ and we write $\Delta(T,B)$ for the
corresponding set of simple roots. In this case, $T$ is a minimal Levi of $G$, and we will call it also a quasi-split maximal torus of $G$. If $T_1,T_2$ are two maximal tori
in $G$ and  $\tx{Ad}(h) : T_1 \rw T_2$ with $h \in G(\ol{F})$ is an isomorphism, then the dual of this isomorphism will be denoted by $\hat{\tx{Ad}}(h) :
\hat{T_2} \rw \hat{T_1}$. The subset of strongly regular semi-simple elements of $G$ will be denoted by $G_\tx{sr}$. We will write $Z^1(F,G)$ or
$Z^1(\Gamma,G)$ for the set of continuous 1-cocycles of $\Gamma$ in the discrete group $G(\ol{F})$, and we will write $Z^1(W_F,G(\ol{L}))$ for the set of
continuous 1-cocycles of $W_F$ in the discrete group $G(\ol{L})$. In \cite{Kot97}, Kottwitz defines a subset $Z^1(W_F,G(\ol{L}))_b$ of "basic" 1-cocycles. If
$u \in Z^1(F,G)$ is an unramified cocycle, we will abuse notation and use the letter $u$ also for the value of $u$ at $\Phi$. Given an element $g \in G$, we
will write $\tx{Cent}(g,G)$ or $G^g$ for the centralizer of $g$ in $G$, and $G_g$ for its connected component.

We will write $\mc{B}^\tx{red}(G,F)$ for the reduced building of $G(F)$, and $\mc{A}^\tx{red}(T,F)$ for the reduced apartment of any maximal split torus $T
\subset G$. For any $x \in \mc{B}^\tx{red}(G,F)$, we let $G(F)_x$ be the fixator of $x$ for the action of $G(F)$ on $\mc{B}^\tx{red}(G,F)$, and $G(F)_{x,0}$
resp. $G(F)_{x,0+}$ be the parahoric subgroup corresponding to $x$ resp. its unipotent radical. If $T$ is an unramified maximal torus of $G$, then it is a
maximal split torus in $G \times F^u$, and hence we have the apartment $\mc{A}^\tx{red}(T,F^u)$ inside $\mc{B}^\tx{red}(G,F^u)$. In this situation, we will
denote $\mc{A}^\tx{red}(T,F^u) \cap \mc{B}^\tx{red}(G,F)$ by $\mc{A}^\tx{red}(T,F)$, even though this will in general not be an apartment in
$\mc{B}^\tx{red}(G,F)$.

If $A$ is a finite abelian group, we write $A^D = \tx{Hom}(A,\C^\times)$. More generally if $A$ is a complex diagonalizable group, we will write $A^D=X^*(A)$.

\section{Extended pure inner twists} \label{sec:eptwists}

\subsection{Definition and basic properties}

We recall some basic facts about isocrystals with additional structure and fix some more notation to be used throughout the rest of the paper. Let $A$ be a connected reductive group defined over $F$. To any element $z \in Z^1(W_F,A(\ol{L}))$ one can define \cite[\S 3.2]{Kot97} a homomorphism of $L$-groups
\[ \nu_z : \mb{D} \rw A, \]
where $\mb{D}$ is the pro-diagonalizable group whose character group is the trivial $\Gamma$-module $\Q$. The conjugacy class of $\nu_z$ depends only on the cohomology class of $z$. We can use $\nu_z$ to define two distinguished subsets of
\[ \tb{B}(A) = H^1(W_F,A(\ol{L})). \]
First, we can consider the set of those classes $[z]$ whose corresponding maps $\nu_z$ are trivial. It is shown in \cite[\S 3.2]{Kot97} that this is precisely the image of the natural injection
\[ H^1(F,A(\ol{F})) \rw \tb{B}(A). \]
Next, we can consider the set of all $z \in Z^1(W_F,A(\ol{L}))$ for which $\nu_z$ factors through $Z(A)$. Those $z$ are called \emph{basic}, their set is denoted by $Z^1(W_F,A(\ol{L}))_b$ and the set of their cohomology classes by $\tb{B}(A)_b$.

Kottwitz has shown that each basic element $z \in Z^1(W_F,A(\ol{L}))$ gives rise to an inner form of $A$. We'd like to recall this process in a slightly different form. Let $E(A)$ be the following pull-back
\begin{diagram}[height=25pt,midshaft,LaTeXeqno] \label{dia:e(a)}
E(A)&&\rTo^{p_2}&Z^1(W_F,A(\ol{L}))_b\\
\dTo<{p_1}&&&\dTo\\
Z^1(\Gamma,A_\tx{ad}(\ol{F}))&&\rInto&Z^1(W_F,A_\tx{ad}(\ol{L}))_b
\end{diagram}

An extended pure inner twist (ep twist) $(\psi,b) : A \rw B$ consists of a a connected reductive group $B$ defined over $F$, an isomorphism $\psi : A \rw B$ of
$\ol{F}$-groups and an element $b \in E(A)$ such that
\[ \psi^{-1}\sigma(\psi) = \tx{Ad}(p_1b(\sigma)) \qquad\forall \sigma \in \Gamma.\]
If $(\psi,b) : A \rw B$ and $(\phi,c) : B \rw C$ are ep twists, then we can form
\[ (\psi,b)^{-1} : B \rw A,\qquad (\phi,c) \circ (\psi,b) : A \rw C \]
by $(\psi,b)^{-1}=(\psi^{-1},\psi(b)^{-1})$ and $(\phi,c)\circ(\psi,b) = (\phi\circ\psi,\psi^{-1}(c)b)$ where multiplication and inversion of elements of $E$
is to be taken component- and pointwise. Two ep twists $(\psi_i,b_i) : A \rw B$ are called equivalent if there exists an element $(x,y) \in A_\tx{ad}(\ol{F})
\times_{A_\tx{ad}(\ol{L})} A(\ol{L})$ such that
\[(\psi_1,b_1)^{-1}\circ(\psi_2,b_2) = (\tx{Ad}(x),(x^{-1}\sigma(x),y^{-1}\sigma(y))). \]
The twist $(\phi,b)$ is called strongly trivial if $b=1$, and trivial if it is equivalent to a strongly trivial twist.

\begin{fct}\ \\[-20pt]
\begin{enumerate}
\item The map $E(A) \rw B(A)_b$ induced by $p_2$ is surjective.
\item If $Z(A)$ is connected, then the map $E(A) \rw H^1(F,A_\tx{ad})$ induced by $p_1$ is surjective.
\end{enumerate}
\end{fct}
\begin{proof} We recall the well-known arguments for the convenience of the reader. Let $[z] \in B(A)_b$. Using Steinberg's theorem we may choose $z$ to be an unramified cocycle $\<\Phi\> \rw A(L)$. Using \cite[(4.3.3)]{Kot85} we may further assume that the value of $z$ at $\Phi^n$, for some suitable integer $n$, belongs to $Z(A)$. Thus the prolongation of $z$ to $A_\tx{ad}$ is trivial on the subgroup of $\<\Phi\>$ generated by $\Phi^n$, which shows that $z$ belongs to $E(A)$.

Assuming now that $Z(A)$ is connected, we need to show that in the diagram
\[ H^1(\Gamma,A_\tx{ad}(\ol{F})) \rw \tb{B}(A_\tx{ad})_b \lw \tb{B}(A)_b \]
the image of the first map is contained in the image of the second map. The first map is a bijection because $Z(A_\tx{ad})=\{1\}$, and the second map is surjective because it is dual to the map $Z(\hat{G}_\tx{sc})^\Gamma \rw Z(\hat G)^\Gamma$ \cite[\S4]{Kot97}, which is injective due to our assumption.
\end{proof}

We can now redo all definitions and arguments of \cite[\S2.1]{Kal09} with the functor $A \mapsto H^1(\Gamma,A(\ol{F}))$  replaced by $A \mapsto \tb{B}(A)_b$. It is straightforward to see that all statements stated there for pure inner twists
remain valid for ep twists: The statement of Fact 2.1.1(1) follows from the next lemma, the statement of Fact 2.1.2(2) is to be interpreted to mean that the new and old definitions of $\tx{inv}$ correspond under the injection $H^1(\Gamma,A_a(\ol{F})) \rw \tb{B}(A_a)_b$, and the other statements remain unchanged.

\begin{lem} Let $a,b \in A(F)$, and let $a_s$ be the semi-simple part of $a$. Put
\begin{eqnarray*}
C(a,b; \ol{L})&=&\{ g \in A(\ol{L})| \tx{Ad}(g)a=b \wedge \forall \sigma \in \Gamma: g^{-1}\sigma(g) \in G_{a_s}(\ol{L}) \},\\
C(a,b; \ol{F})&=&\{ g \in A(\ol{F})| \tx{Ad}(g)a=b \wedge \forall \sigma \in \Gamma: g^{-1}\sigma(g) \in G_{a_s}(\ol{F}) \}.\\
\end{eqnarray*}\ \\[-30pt]
Then $C(a,b;\ol{L}) \neq \emptyset \Leftrightarrow C(a,b;\ol{F}) \neq \emptyset$.
\end{lem}
\pf For the duration of this proof, we will say that $a,b$ are $\ol{L}$-stably conjugate if the first set is non-empty, and $\ol{F}$-stably conjugate if the
second set is non-empty.

Assume first that $A$ has a simply-connected derived group. Then $G^a \subset G^{a_s} = G_{a_s}$ and the second conditions in the definitions of the above sets
are vacuous. The sets themselves are then the sets of $\ol{L}$-points resp. the $\ol{F}$-points of the algebraic variety $\{g \in A| \tx{Ad}(g)a=b \}$. This
variety is defined over $F$ and the statement follows.

We reduce the general case to this by taking a $z$-extension $A' \rw A$. The argument of the proof of \cite[3.1.(2)]{Kot82} shows that $a,b$ are
$\ol{L}$-stably conjugate if and only if there exist lifts $a',b' \in A'(F)$ which are $\ol{L}$-stably conjugate. Applying the proved special case this is
equivalent to $a',b'$ being $\ol{F}$-stably conjugate, which again by loc. cit. is equivalent to $a,b$ being $\ol{F}$-stably conjugate.\qed

By a representation of an ep twist of $A$ we shall mean a quadruple $(B,\psi,b,\pi)$ where $(\psi,b) : A \rw B$ is an ep twist and $\pi$ is a representation of
$B(F)$. Two such quadruples $(B,\psi,b,\pi)$ and $(B',\psi',b',\pi')$ will be called equivalent if there exists a
strongly-trivial ep twist $B \rw B'$ which is equivalent to $\psi'\psi^{-1}$ identifies $\pi$ with $\pi'$. It is easy to check that then every strongly-trivial ep twist $B \rw B'$ which is equivalent to $\psi'\psi^{-1}$ and identifies $\pi$ with $\pi'$

\subsection{Transfer factors} \label{sec:tf}

Let $G$ be a connected reductive group, defined and quasi-split over $F$, $(\psi,b) : G \rw G'$ be a ep twist, and $(H,s,{^L\eta})$ be an extended endoscopic triple for $G$. Fix a normalization $\Delta_G^H$
of the absolute transfer factor for $(G,H)$. We are going to define a normalization $\Delta_{G'}^H$ of the absolute transfer factor for $(G',H)$ as follows.
Let $\gamma' \in G'(F)$ and $\gamma^H \in H(F)$ be a pair of related strongly $G$-regular elements. Choose any $\gamma \in G(F)$ stably-conjugate to $\gamma'$,
and let $\phi_{\gamma,\gamma^H} : T_\gamma \rw T^H_{\gamma_H}$ be the unique admissible isomorphism from the centralizer of $\gamma$ in $G$ to the centralizer
of $\gamma^H$ in $H$. In \cite[\S2]{Kot85}, Kottwitz constructs an isomorphism $X_*(T_{\gamma})_\Gamma \rightarrow \tb{B}(T_{\gamma})$, and this isomorphism
provides a pairing
\[ \langle\ ,\ \rangle : \tb{B}(T_{\gamma}) \times \hat{T_{\gamma}}^\Gamma \rw \C^\times. \]
We put
\[ \Delta_{G'}^H(\gamma^H,\gamma') = \Delta_G^H(\gamma^H,\gamma) \langle \tx{inv}(\gamma,\gamma'), \hat\phi_{\gamma,\gamma^H}(s) \rangle^{-1}. \]
The same proof as for \cite[Lemma 2.2.1]{Kal09} shows that $\Delta_{G'}^H$ is well-defined. Moreover, using \cite[2.7]{Kot85} and \cite[Lemma 2.3.2]{Kal09} one
sees that if $(\psi,b)$ is a pure inner twist, then this definition agrees with the one given in \cite[\S2.2]{Kal09}.

\begin{pro} \label{pro:tf} Assume that $Z(G)$ is connected. Then $\Delta_{G'}^H$ is an absolute transfer factor for $(G',H)$.
\end{pro}
\pf For $i=1,2$, let $\gamma_i' \in G'(F)$ and $\gamma^H_i \in H(F)$ be a pair of strongly $G$-regular related elements, and let $\gamma_i \in G(F)$ be stably
conjugate to $\gamma_i'$. As in the proof of \cite[Proposition 2.2.2]{Kal09} one reduces to showing that the following equality holds:
\begin{equation} \label{eq:delta_1}
\frac{\langle\tx{inv}(\gamma_1,\gamma_1'),\hat\phi_{\gamma_1,\gamma_1^H}(s)\rangle^{-1}}{\langle\tx{inv}(\gamma_2,\gamma_2'),\hat\phi_{\gamma_2,\gamma_2^H}(s)\rangle^{-1}}=
   \left\langle \tx{inv}\left(\frac{\gamma_1,\gamma_1'}{\gamma_2,\gamma_2'}\right),s_U \right\rangle.
\end{equation}
We put $T_i = \tx{Cent}(\gamma_i,G)$, $T_i^H = \tx{Cent}(\gamma_i^H,H)$, and form the push-out diagrams
\begin{diagram}[height=20pt]
Z(G)&\rTo&T_1&&Z(G_\tx{sc})&\rTo&[T_1]_\tx{sc}\\
\dTo&&\dTo&&\dTo&&\dTo\\
T_2&\rTo&V&&[T_2]_\tx{sc}&\rTo&U
\end{diagram}
The canonical map
\[ [T_1]_\tx{sc} \times [T_2]_\tx{sc} \rightarrow T_1 \times T_2 \rightarrow V \]
factors uniquely through the isogeny $[T_1]_\tx{sc} \times [T_2]_\tx{sc} \rightarrow U$ to give a map $U \rightarrow V$. One sees easily that the image of
$\tx{inv}(\gamma_1,\gamma_1'/\gamma_2,\gamma_2')$ under the map $H^1(\Gamma,U) \rw H^1(\Gamma,V) \rw \tb{B}(V)$ coincides with the image of
$(\tx{inv}(\gamma_1,\gamma_1')^{-1},\tx{inv}(\gamma_2,\gamma_2'))$ under the canonical map $\tb{B}(T_1) \times \tb{B}(T_2) \rw \tb{B}(V)$. Equation
\eqref{eq:delta_1} will follow from \cite[2.7]{Kot85} and the functoriality of the Tate-Nakayama pairing once we exhibit an element $s_V \in \hat V^\Gamma$,
whose image under $\hat V \rw \hat U$ equals $s_U$, and whose image under $\hat V \rw \hat T_1 \times \hat T_2$ equals
$(\hat\phi_{\gamma_1,\gamma_1^H}(s),\hat\phi_{\gamma_2,\gamma_2^H}(s))$.

By Langlands duality we obtain a presentation of $\hat V$ as the pull-back of the dual diagram
\[ \hat T_1 \rightarrow \hat G/\hat G_\tx{der} \leftarrow \hat T_2. \]
The image of $s$ under
\begin{diagram}[midshaft]
Z(\hat H)&&&\rTo^{\Delta}&\hat{T^H_1} \times \hat{T^H_2}&&&\rTo^{\hat\phi_{\gamma_1,\gamma_1^H},\hat\phi_{\gamma_2,\gamma_2^H}}&\hat{T_1} \times \hat{T_2}
\end{diagram}
belongs to $\hat{V}$, call it $s_V$. Consider the diagrams
\begin{center}
\begin{diagram}[height=20pt,inline]
[T_1]_\tx{sc} \times [T_2]_\tx{sc}&\rTo&T_1 \times T_2\\
\dTo&&\dTo\\
U&\rTo&V
\end{diagram}\hspace{20pt}
\begin{diagram}[height=20pt,inline]
\hat{[T_1]}_\tx{ad} \times \hat{[T_2]}_\tx{ad}&\lTo&\hat{T_1} \times \hat{T_2}\\
\uTo&&\uTo\\
\hat{U}&\lTo&\hat{V}
\end{diagram}
\end{center}

The map $U \rw V$ is characterized uniquely by the property that it makes the left diagram commutative, hence its dual is the unique map $\hat V \rw \hat U$
making the right diagram commutative. Explicitly, this map is given as follows. Recall from \cite[\S3.4]{LS87} that $\hat{U} = [\hat{T_1}]_\tx{sc} \times
[\hat{T_2}]_\tx{sc} / Z(\hat{G}_\tx{sc})$ where $Z(\hat{G}_\tx{sc})$ is embedded diagonally into $[\hat{T_1}]_\tx{sc} \times [\hat{T_2}]_\tx{sc}$. Let
$(t_1,t_2) \in \hat{T_1} \times \hat{T_2}$ be an element of $\hat{V}$. Thus $t_1$ and $t_2$ have the same image in $\hat G/\hat G_\tx{der}$. Choose $z \in
Z(\hat G)$ mapping to that image. Then $(z^{-1}t_1,z^{-1}t_2)$ belongs to $[\hat{T_1}]_\tx{sc} \times [\hat{T_2}]_\tx{sc}$. The image of
$(z^{-1}t_1,z^{-1}t_2)$ in $\hat{U}$ is independent of the choice of $z$. The map $\hat{V} \rightarrow \hat{U}$ obtained in this way obviously makes the right
diagram commute and thus must be the dual of $U \rw V$. From the explicit description of the map $\hat V \rw \hat U$ one reads off that the image of $s_V$
equals $s_U$, and this concludes the proof of the proposition. \qed

\section{Construction of packets} \label{sec:clp}

Throughout this section, we assume that $p$ is odd.

\subsection{The representations} \label{sec:reps}

Let $G$ be a connected reductive group defined over $F$. Consider a pair $(T,\theta)$, where $T \subset G$ is an elliptic unramified maximal torus, and $\theta$ is a
character $\theta : T(F) \rw \C^\times$, which is of depth-zero and regular. Recall that a character of $T(F)$ is of depth-zero if its restriction to the pro-p-Sylow subgroup of the maximal compact subgroup of $T(F)$ is trivial, and it is regular if its stabilizer in $\Omega(T,G)$ be trivial.

To such a pair we assign a smooth irreducible supercuspidal representation $\pi_{G,T,\theta}$ of $G(F)$, as was done in \cite[\S4.4]{DR09}. Let us briefly recall its construction. According to \cite[\S2.2]{Deb06}, $x:=\mc{A}^\tx{red}(T,F)$ is a vertex. Let $\ts{G}$ be the connected reductive group defined over $k_F$ associated to this vertex, and $\ts{T}$ be
the maximal torus in $\ts{G}$ corresponding to $T$. The character $\theta$ gives rise to a regular character $\theta : \ts{T}(k_F) \rw \C^\times$. By
\cite[8.3]{DL76} one obtains an irreducible cuspidal representation $\epsilon(\ts{G},\ts{T})R_T^\theta$ of $\ts{G}(k_F)$. Let $\kappa^0$ denote the inflation of this representation to $G(F)_{x,0}$, and put
\[ \pi_{G,T,\theta} = \tx{Ind}_{Z(F)G(F)_{x,0}}^{G(F)_{x}} \left(\theta\otimes\kappa^0\right). \]
According to \cite[Lemma 4.5.1]{DR09}, this representation is irreducible and supercuspidal.

\begin{lem} \label{lem:pairiso}
The representations $\pi_{G,T_1,\theta_1}$ and $\pi_{G,T_2,\theta_2}$ are isomorphic if and only if the pairs $(T_1,\theta_1)$, $(T_2,\theta_2)$ are
$G(F)$-conjugate
\end{lem}\ \\[-20pt]
\pf Conjugate pairs give rise to conjugate inducing data and hence to isomorphic representations. Conversely assume that $\pi_{G,T_1,\theta_1} \cong
\pi_{G,T_2,\theta_2}$. Let $x,y \in \mc{B}^\tx{red}(G,F)$ be the vertices corresponding to $T_1,T_2$. According to \cite[3.5]{MP96}
the depth-zero unrefined minimal $K$-types $(G(F^u)_{x,0},\kappa^0_1)$ and $(G(F^u)_{y,0},\kappa^0_2)$ are associate, hence there exists $g \in G(F)$ such that
$gx=y$ and $\tx{Ad}(g)\kappa^0_1=\kappa^0_2$. Thus we may assume $x=y$ and $\kappa^0_1=\kappa^0_2$. Using \cite[6.8]{DL76} we find $g \in G(F)$ such that
$\tx{Ad}(g)(\ts{T}_1,\theta_1)=(\ts{T}_2,\theta_2)$. The lemma now follows from \cite[2.2.2]{Deb06}.\qed

\begin{lem} \label{lem:1char} Let $\Theta$ be the character of $\pi_{G,T,\theta}$ as a function on $G_\tx{sr}(F)$.
If $Q_T \in \tx{Lie}(T)(F)$ is a fixed regular semi-simple element, then for any $z \in Z_G(F)$ and $\gamma \in G_\tx{sr}(F)_0$ we have
\[ \Theta(z\gamma) = \epsilon(G,A_G)\theta(z)\sum_{Q} R(G_{\gamma_s},S_Q,1)(\gamma_u)[\phi_{Q_T,Q}]_*\theta(\gamma_s), \]
where $Q$ runs over any set of representatives for the $G_{\gamma_s}(F)$-conjugacy classes inside the $G(F)$-conjugacy class of $Q_T$, and $S_Q =
\tx{Cent}(Q,G)$.
\end{lem}
\pf The proof is the same as for \cite[6.2.1]{Kal09}.\qed

\subsection{The parameters} \label{sec:parms}

Let $G$ be an unramified connected reductive group defined over $F$, $^LG= \hat G \rtimes W_F$ its $L$-group, and $[\phi] : W_F \rw {^LG}$ an equivalence class of Langlands parameters. If this class
satisfies certain conditions, we are going to construct in the next section a triple $(S_0,[a],[^Lj])$ where $S_0 \subset G$ is an unramified maximal torus,
$[a] : W_F \rw {^LS_0}$ is an equivalence class of Langlands parameters, and $[^Lj] : {^LS_0} \rw {^LG}$ is an $\hat G$-conjugacy class of unramified
$L$-embeddings such that $[^Lj]\circ[a] = [\phi]$. The parameter $[a]$ gives rise to a character $\theta_0 : S_0(F) \rw \C^\times$.

Choose a representative $\phi$ within the equivalence class $[\phi]$, and let $\phi_0 : W_F \rw \hat G$ be the
composition of $\phi$ with the projection to $\hat G$. Then $\phi_0$ is an element of $Z^1(W_F,\hat G)$ whose cohomology class is independent of the choices made. The conditions we are imposing on $[\phi]$ are the following:
\begin{enumerate}
\item $\phi_0$ restricts trivially to the wild inertia subgroup of $W_F$,
\item the centralizer of $\phi_0(I_F)$ in $\hat G$ is a maximal torus,
\item the index of $Z(\hat G)^\Gamma$ in $\tx{Cent}(\phi,\hat G)$ is finite.
\end{enumerate}

These are the Langlands parameters considered in \cite{DR09}. For the construction of the triple $(S_0,[a],[{^Lj}])$, conditions 2 and 3 are enough, thus it is applicable to regular elliptic parameters of positive depth as well. Condition 1 is used in the construction of the representations.

\begin{pro} Assume that $[\phi]$ satisfies conditions 1-3 above, and let $\theta_0$ and $(S_0,[a],[^Lj])$ be as before. Then $\theta_0$ is a regular
character of depth $0$.
\end{pro}
\pf Let $w \in \Omega(S_0,G)(F)$. By functoriality of the Langlands correspondence we have $w\theta_0=\theta_0$ if and only if $[a]=[w \circ a]$. If $a$ is any
representative of the class of $a$, then a necessary condition for the second equality is $a|_{I_F}=w\circ a|_{I_F}$ (recall that $S_0$ is unramified). This
however is precluded by condition $2$ on $\phi$, and we conclude that $\theta_0$ is regular. Moreover, since the Langlands correspondence preserves depth,
$\theta_0$ is of depth $0$ (see \cite[\S4]{Ree08} for an explicit version of the Langlands correspondence for unramified tori).
\qed

\subsection{The packets} \label{sec:packs}

We continue with $G$, $^LG$ and $[\phi]$ as in the previous section, and require now that $[\phi]$ satisfies the properties 1-3 described there. For each
$\phi_1 \in [\phi]$, let $S_{\phi_1}$ be the centralizer of $\phi_1$ in $\hat G$ (note that it is abelian). For $\phi_1,\phi_2 \in [\phi]$ there is a canonical isomorphism $S_{\phi_1}
\rw S_{\phi_2}$, and we define $S_{[\phi]} = \underrightarrow\lim S_{\phi_1}$, where $\phi_1$ runs over $[\phi]$. In this section we are going to construct a
set $\Pi_{[\phi]}$ of (equivalence classes of) irreducible admissible representations of ep twists of $G$ and an explicit bijection
\begin{equation} S_{[\phi]}^D \rw \Pi_{[\phi]}. \label{eq:packs} \end{equation}
This bijection is going to depend on a choice of a hyperspecial vertex $o \in \mc{B}^\tx{red}(G,F)$. The left-hand side of the above map is a group and thus
has a distinguished element, namely the identity, while the right-hand side is a-priori just a set. The choice of $o$ serves in particular to fix a base-point
in this set, and different choices will lead to different base-points. The set $\Pi_{[\phi]}$ itself will be independent of the choice of $o$. Moreover, we
will have the diagram
\begin{diagram}[small]
S_{[\phi]}^D&\rTo&\Pi_{[\phi]}\\
\dTo&&\dTo\\
[Z(\hat G)^\Gamma]^D&\rTo&\tb{B}(G)_b
\end{diagram}
where the left vertical map is the dual of the natural inclusion $Z(\hat G)^\Gamma \rw S_{[\phi]}$ and the right vertical map sends an irreducible
representation to the ep twist it lives on.

Fix a hyperspecial vertex $o \in \mc{B}^\tx{red}(G,F)$, and endow $G$ with the corresponding $O_F$-structure. The construction of \eqref{eq:packs} proceeds in
two steps. In the first step, we will explicitly construct a triple $(S_0,[a],[^Lj])$, where $S_0$ is an elliptic maximal torus of $G$ defined over $O_F$,
$[a]$ is an equivalence class of Langlands parameters $a : W_F \rw {^LS_0}$, and $[^Lj]$ is a $\hat G$-conjugacy class of unramified $L$-embeddings $^Lj :
{^LS_0} \rw {^LG}$ such that $[^Lj] \circ [a]=[\phi]$. There will be multiple choices involved in the construction, but we will see in the next section that
such a triple is essentially unique, so there will be no need to keep track of these choices. As already mentioned, this step does not require condition 1 of the Langlands parameter, and can be executed for parameters satisfying only conditions 2 and 3.  In the second step, we will construct to each $\lambda \in
X_*(S_0)$ a quadruple $(G^\lambda,\psi_\lambda,b_\lambda,\pi_\lambda)$, where $(\psi_\lambda,b_\lambda) : G \rw G^\lambda$ is an ep twist and $\pi_\lambda$ is
an irreducible depth-zero supercuspidal representation of $G^\lambda(F)$. In this step, condition 1 of the Langlands parameter is used. We will show that the quadruples associated to $\lambda,\mu \in X_*(S_0)$ are equivalent if and
only $\lambda,\mu$ have the same image in $X_*(S_0)_\Gamma$. The set of equivalence classes of these quadruples will be the $L$-packet $\Pi_{[\phi]}$. It is by
construction in bijection with $X_*(S_0)_\Gamma$, and we obtain \eqref{eq:packs} by composing this bijection with the group isomorphism $S_{[\phi]}^D \rw
X_*(S_0)_\Gamma$ given by $[^Lj]$.

\ul{Step 1}: Let $T \subset G$ be a quasi-split maximal torus such that $o \in \mc{A}(T,F)$. Choose a
$W_F$-invariant maximal torus $\hat T \subset \hat G$, which is in duality with $T$ and is part of a $\Gamma$-invariant splitting of $\hat G$. Choose a
representative $\phi : W_F \rw {^LG}$ of $[\phi]$ so that $\phi_0(I_F) \subset \hat T$, where $\phi_0$ is the composition of $\phi$ with the projection $^LG
\rw \hat G$. Then $\phi_0(W_F) \subset N(\hat G,\hat T)$ and we obtain an unramified cocycle
\[ w : W_F \rw N(\hat G,\hat T) \thrw \Omega(\hat T,\hat G) \cong \Omega(T,G), \]
which extends continuously to $\Gamma$.

\begin{lem} There exists $p_0 \in G(O_{F^u})$ with the properties that the maximal torus $S_0:= \tx{Ad}(p_0)T$ is defined over $O_F$ and
\[ \tx{Ad}(p_0) : T^w \rw S_0 \]
is an isomorphism of $O_F$-tori, where $T^w$ is the twist of the $O_F$-torus $T$ by $w$.
\end{lem}
\pf We follow the argument of the proof of \cite[Prop. 5.1.10]{BT2}. Write $\ts{G}$ and $\ts{T}$ for the special fibers of $G$ and $T$. The reduction map $G(O_{F^u}) \rw \ts{G}(\ol{k_F})$ provides an isomorphism $\Omega(T,G)
\cong \Omega(\ts{T},\ts{G})$ and the special fiber of $T^w$ is the same as the twist of $\ts{T}$ by the image of $w$ under this map. Choose a $k_F$-embedding
$\bar\xi : \ts{T}^w \rw \ts{G}$. By \cite[exp. 11, Cor 4.2]{SGA3}, the functor from $O_F$-algebras to sets given by $R \mapsto \tx{Hom}_{R\tx{-grp}}(T^w \times R,G \times R)$ is
representable by a smooth $O_F$-scheme, hence the map $\tx{Hom}_{O_F\tx{-grp}}(T^w,G) \rw \tx{Hom}_{k_F\tx{-grp}}(\ts{T}^w,\ts{G})$ is surjective and we can lift
$\bar\xi$ to an $O_F$-homomorphism $\xi : T^w \rw G$, which by \cite[exp. 9, Cor. 2.5 and 6.6]{SGA3} is a closed embedding. Put $S_0 = \xi(T^w)$, and let $\bar p_0 \in \ts{G}(\ol{k_F})$ be such that $\tx{Ad}(\bar p_0) : \ts{T}^w \rw
\ts{S}_0$ is an isomorphism of $k_F$-tori. The scheme $\{g \in G| \tx{Ad}(g)T = S_0 \}$ is smooth over $O_{F^u}$ and hence we can lift $\bar p_0$ to $p_0 \in
G(O_{F^u})$. Then the map $\sigma \mapsto p_0^{-1}\sigma(p_0)$ gives an element of $Z^1(\Gamma,\Omega(T,G))$ which equals $w$, because its image in
$Z^1(\Gamma,\Omega(\ts{T},\ts{G}))$ equals by construction the image of $w$. \qed

Choose unramified $\chi$-data for $R(S_0,G)$. By the construction of \cite[2.6]{LS87} we obtain a $\hat G$-conjugacy class of $L$-embeddings ${^LS}_0 \rw
{^LG}$. Let $^Lj$ be an element of this class with the following properties:
\begin{itemize}
\item $^Lj(\hat S_0) = \hat T$
\item The cocycles
\[ W_F \rTo^{\phi} N(\hat T,\hat G) \rtimes W_F \rOnto \Omega(\hat T,\hat G) \]
and
\[W_F \rTo^{^Lj} N(\hat T,\hat G) \rtimes W_F \rOnto \Omega(\hat T,\hat G) \]
are equal
\end{itemize}
Then the image of $\phi$ is contained in the image of $^Lj$ and thus we obtain a factorization
\begin{diagram}[height=20pt]
^LS_0&\rTo&^LG\\
\uTo<{a}&\ruTo>{\phi}\\
W_F
\end{diagram}

This completes the construction of the triple $(S_0,[a],[^Lj])$. The Langlands parameter $[a]$ gives rise to a character $\theta_0 : S_0(F) \rw \C^\times$.

\ul{Step 2:} Let $\lambda \in X_*(S_0)$. We are going to construct the quadruple $(G^\lambda,\psi_\lambda,b_\lambda,\pi_\lambda)$ as follows: The map $\Phi
\mapsto \lambda(\pi)$ extends to a 1-cocycle $W_F \rw S_0(\ol{L})$. Its prolongation to $G(\ol{L})$ is basic and will be called $b_\lambda$. We claim
$b_\lambda \in E(G)$. Let $E/F$ be a finite unramified extension splitting $S_0$, $\Sigma = \tx{Gal}(E/F)$, $n=|\Sigma|$. Recall that $S_0$ is elliptic, and
thus $X_*(S_0)^\Gamma \subset X_*(Z)^\Gamma$, which implies
\[ b_\lambda(\Phi^n) = N_\Sigma(\lambda(\pi)) = [N_\Sigma(\lambda)](\pi) \in Z(F), \]
where $N_\Sigma$ is the norm map for the action of $\Sigma$ on $X_*(S_0)$.
This shows that the image of $b_\lambda$ in $Z^1(W_F,G_\tx{ad}(\ol{L}))$ factors through $\Sigma$ and thus belongs to the image of
$Z^1(\Gamma,G_\tx{ad}(\ol{F}))$, which proves the claim $b_\lambda \in E(G)$. Now $b_\lambda$ gives rise to the ep twist $(G^\lambda,\psi_\lambda,b_\lambda)$.
Put $S_\lambda := \psi_\lambda(S_0)$. Since $\tx{Ad}(b_\lambda)$ acts trivially on $S_0$, we see that
\[ \psi_\lambda : S_0 \rw S_\lambda \]
is an isomorphism of $F$-tori. Put $\theta_\lambda = [\psi_\lambda]_*\theta_0$, and let $\pi_\lambda$ be the representation
$\pi_{G^\lambda,S_\lambda,\theta_\lambda}$ defined in \ref{sec:reps}.

\begin{lem} The quadruples $(G^\lambda,\psi_\lambda,b_\lambda,\pi_\lambda)$, $(G^\mu,\psi_\mu,b_\mu,\pi_\mu)$ constructed from two elements
$\lambda,\mu \in X_*(S_0)$ are equivalent if and only if  $\lambda$ and $\mu$ have the same image in $X_*(S_0)_\Gamma$.
\end{lem}
\pf By Lemma \ref{lem:pairiso} the two quadruples are equivalent if and only if there exists a strongly-trivial ep twist $G^\lambda \rw G^\mu$ which is
equivalent to $(\psi_\mu,b_\mu)\circ(\psi_\lambda,b_\lambda)^{-1}$ and identifies the pairs $(S_\lambda,\theta_\lambda)$ and $(S_\mu,\theta_\mu)$.

Assume first that $\lambda,\mu$ have the same image in $X_*(S_0)_\Gamma$. Using Kottwitz's isomorphism $X_*(S_0)_\Gamma \rw \tb{B}(S_0)$ we see that there
exists $t \in S_0(L)$ such that $b_\lambda = t^{-1}b_\mu\Phi(t)$. Then
\[ (\psi_\mu,b_\mu)\circ(\tx{Ad}(t),t^{-1}\Phi(t))\circ(\psi_\lambda,b_\lambda)^{-1} \]
is a strongly-trivial twists $G^\lambda \rw G^\mu$ which by construction identifies the pairs $(S_\lambda,\theta_\lambda)$ and $(S_\mu,\theta_\mu)$.

Conversely let $g \in G(L)$ be such that
\[ (\psi_\mu,b_\mu)\circ(\tx{Ad}(g),g^{-1}\Phi(g))\circ(\psi_\lambda,b_\lambda)^{-1} \]
is a strongly-trivial twist which identifies $(S_\lambda,b_\lambda)$ with $(S_\mu,b_\mu)$. This implies that $\tx{Ad}(g)$ leaves the pair $(S_0,\theta_0)$
invariant, which by the regularity of $\theta_0$ is equivalent to $g \in S_0(L)$. Thus $b_\lambda$ and $b_\mu$ have the same image in $\tb{B}(S_0)$, and using
again the Kottwitz isomorphism we conclude that $\lambda,\mu$ have the same image in $X_*(S_0)_\Gamma$.\qed

We let $\Pi_{[\phi]}$ be the set of equivalence classes of quadruples just constructed. Choose any representative $^Lj$ within the given $\hat G$-conjugacy
class, let $\phi = {^Lj}\circ a$ and put $S_\phi = \tx{Cent}(\phi,\hat G)$. The map $^Lj$ identifies $[\hat S_0]^\Gamma$ with $S_\phi$, and hence we obtain a
bijection
\[ S_{[\phi]}^D \rw S_\phi^D \rw X_*(S_0)_\Gamma \rw \Pi_{[\phi]} \]
which is obviously independent of the choice of $^Lj$.

\subsection{Independence of choices and change of base points} \label{sec:indep}

This section has a two-fold purpose. On the one hand, we will see that the $L$-packet $\Pi_{[\phi]}$ is independent of all choices, while the bijection
\eqref{eq:packs} depends only on the $G(F)$-orbit of $o$. On the other hand, given two choices of hyperspecial vertices $o,o' \in \mc{B}^\tx{red}(G,F)$, we will give the precise relationship between the corresponding bijections.

Let $[\phi]$ be a equivalence class of Langlands parameters satisfying conditions 2 and 3 of Section \ref{sec:parms}. Let $\Xi([\phi],o)$ be the set of triples $(S_0,[a],[{^Lj}])$ where
\begin{itemize}
\item $S_0 \subset G$ is an elliptic maximal torus defined over $O_F$,
\item $[a] : W_F \rw {^LS_0}$ is an equivalence class of Langlands parameters, and
\item $[^Lj] : {^LS_0} \rw {^LG}$ is a $\hat G$-conjugacy class of unramified $L$-embeddings,
\end{itemize}
subject to the conditions
\begin{itemize}
\item $[^Lj]\circ [a] = [\phi]$, and
\item the following diagram commutes
\begin{diagram}[small,midshaft]
^LS_0&&&\rTo^{^Lj}&^LG\\
&\rdTo&&\ldTo\\
&&^LZ(G)^\circ
\end{diagram}
where the diagonal arrows are dual to the natural inclusions.
\end{itemize}

The triple constructed in Step 1 of the previous section is an element of $\Xi([\phi],o)$.

If $u$ is an inner automorphism of $G$ defined over $F$, then we get a map
\[ \Xi([\phi],o) \rw \Xi([\phi],uo),\quad (S_0,[a],[^Lj]) \mapsto (u(S_0),[\hat u^{-1}\circ a],[^Lj\circ \hat u]). \]
In particular, the group $G(O_F)$ acts on the set $\Xi([\phi],o)$ by conjugation.

\begin{lem} \label{lem:tripuni} Let $(S_0,[a],[{^Lj}])$ and $(S_0',[a'],[{^Lj'}])$ be two elements of $\Xi([\phi],o)$. Then there exists $h \in G(O_F)$
such that $(S_0',[a'],[{^Lj'}])=\tx{Ad}(h)(S_0,[a],[^Lj])$.
\end{lem}
\pf Choose representatives $a$, $a'$, $^Lj$ and $^Lj'$ within the given conjugacy classes such that $^Lj\circ a={^Lj'}\circ a'$. We claim that the images of
$^Lj$ and $^Lj'$ coincide. On the one hand, we have $^Lj(\hat S_0) = \tx{Cent}({^Lj}\circ a(I_F),\hat G)$ and thus we see that $^Lj(\hat{S_0}) =
{^Lj'}(\hat{S_0'})$. On the other hand, for any $w \in W_F$ we have $^Lj(w) \in {^Lj}(\hat S_0) \cdot {^Lj}\circ a(w)$ and the claim follows. Hence we may
consider
\[ ^L\iota : {^LS_0} \rw {^LS_0'} \]
given by $^Lj'^{-1}\circ{^Lj}$. The restriction of $^L\iota$ to $\hat S_0$ is a $W_F$-equivariant admissible isomorphism $\hat S_0 \rw \hat S_0'$, hence there
exists $h \in G(\ol{F})$ such that $\tx{Ad}(h) : S_0 \rw S_0'$ is defined over $F$ and $^L\iota|_{\hat S_0} = \hat{\tx{Ad}}(h)^{-1}\rtimes 1$. This $h$ is not
yet the element we are looking for. We will successively modify $h$ by multiplying it on the right by elements of $S_0(\ol{F})$ to achieve $h \in G(O_F)$.
First, let $h' \in G(O_{F^u})$ be such that $\tx{Ad}(h')S_0=S_0'$ \cite[exp 12, 1.7]{SGA3}. Since every element of $\Omega(S_0,G)(\ol{F})$ has a representative
in $N(S_0,G)(O_{F^u})$ we see that $h'^{-1}h \in N(S_0,G)(O_{F^u})S_0(\ol{F})$, and conclude that $h \in G(O_{F^u})S_0(\ol{F})$. Modifying $h$ by an element of
$S_0(\ol{F})$ we may assume $h \in G(O_{F^u})$. For each $\sigma \in \Gamma$, $h^{-1}\sigma(h)$ belongs to $N(S_0,G)(O_{F^u})$ and acts trivially on $S_0$,
hence belongs to $S_0(O_{F^u})$. Since $H^1(\Gamma,S_0(O_{F^u}))$ is trivial, we can again modify $h$ by an element of $S_0(O_{F^u})$ and achieve $h \in
G(O_F)$.

We claim that this $h$ satisfies the statement of the lemma. We already know that $\tx{Ad}(h)S_0=S_0'$, and it will be enough to show that ${^Lj'} =
{^Lj}\circ\hat{\tx{Ad}}(h)$. Put $^Lj'' = {^Lj'}\circ\hat{\tx{Ad}}(h)^{-1}$. By construction, the restriction of ${^Lj}^{-1}\circ {^Lj''}$ to $\hat S_0$ is the
natural inclusion $\hat S_0 \rw {^LS_0}$. Let $c : W_F \rw {^LS_0}$ be the restriction of ${^Lj}^{-1}\circ{^Lj''}$ to $W_F$. This is an unramified Langlands
parameter, and we wish to show that the corresponding character $\chi_c$ of $S_0(F)$ is trivial. This will imply that $^Lj''$ and $^Lj$ are conjugate under
$^Lj(\hat S_0)$, hence $[^Lj'']=[^Lj]$ and the proof of the lemma will be complete. Since $\chi_c$ is unramified, it is enough by \cite[7.1.1]{Kal09} to show
that it is trivial when restricted to $Z(G)^\circ(F) \subset S_0(F)$. If $^Lp : {^LS_0} \rw {^LZ(G)^\circ}$ is the dual of the natural inclusion $Z(G)^\circ
\rw S_0$, then the Langlands parameter for $\chi_c|_{Z(G)^\circ(F)}$ is $^Lp \circ c$. On the other hand, if $^Lq : {^LG} \rw {^LZ(G)^\circ}$ is the dual of
the natural inclusion $Z(G)^\circ \rw G$, then we obtain, using the third property of our triples, that
\[^Lp \circ c = {^Lq}\circ{^Lj}\circ c = {^Lq} \circ {^Lj''} = {^Lp}. \]
\qed

We are now going to construct for every pair $o,o'$ of hyperspecial vertices in $\mc{B}^\tx{red}(G,F)$ an element $(o,o') \in S_{[\phi]}$. Choose an element
$(S_0,[a],[^Lj])$ of $\Xi([\phi],o)$. We obtain a group homomorphism
\[ G_\tx{ad}(F) \rw H^1(F,Z(G)) \rw H^1(F,S_0) \rw \tb{B}(S_0) \rw X_*(S_0)_\Gamma \rw S_{[\phi]}^D, \]
where the first map is the connecting homomorphism, the second is induced by the natural inclusion, the third is the injection from \cite[1.8]{Kot85}, the
fourth is the isomorphism constructed in \cite[\S2]{Kot85}, and the fifth is the map induced by $^Lj$. We let $(o,o') \in S_{[\phi]}^D$ be the image of any
$g_\tx{ad} \in G_\tx{ad}(F)$ such that $g_\tx{ad}o=o'$ under this homomorphism. It is clear that if $o' \in G(F)o$, then $(o,o')=1$.

\begin{lem} \label{lem:hyper} This map is well defined, and if $o,o',o'' \in \mc{B}^\tx{red}(G,F)$ are
three hyperspecial vertices, then
\[ (o,o'')=(o,o')\cdot(o',o''). \]
\end{lem}
\pf The independence of the choice of $(S_0,[a],[^Lj])$ follows from Lemma \ref{lem:tripuni}. To show that the choice of $g_\tx{ad}$ is also irrelevant, let $g'_\tx{ad} \in G_\tx{ad}(F)$ be another element with $g'_\tx{ad}o=o'$, and consider $h_\tx{ad}:=g_\tx{ad}^{-1}g_\tx{ad}'$. Endow $G$ and $G_\tx{ad}$ with the $O_F$-structure corresponding to $o$. Then $S_0' = \tx{Ad}(h_\tx{ad})S_0$ is another maximal torus of $G$ defined over $O_F$. Let $h \in G(\ol{F})$ be a lift of $h_\tx{ad}$.  Arguing as in the proof of Lemma \ref{lem:tripuni} we can write $h=h_1 s_1$ where $h_1 \in G(O_F)$ and $s_1 \in S_0(\ol{F})$. Then the image of $h_\tx{ad}$ in $H^1(F,S_0)$ is represented by the cocycle $\sigma \mapsto s_1^{-1}h_1^{-1}\sigma(h_1)\sigma(s_1)$. The cocycle $\sigma \mapsto h_1^{-1}\sigma(h_1)$ takes values in $S_0(O_{F^u})$, and is thus cohomologically trivial. We conclude that the elements
$g_\tx{ad}$ and $g'_\tx{ad}$ map to the same element of $S_{[\phi]}^D$.

The multiplicative property follows easily from the fact that if $(S_0,[a],[^Lj]) \in
\Xi([\phi],o)$ and $g_\tx{ad} \in G_\tx{ad}(F)$ is such that $g_\tx{ad}o=o'$, then $\tx{Ad}(g_\tx{ad})(S_0,[a],[^Lj]) \in \Xi([\phi],o')$ and the maps
$H^1(F,Z(G)) \rw S_{[\phi]}^D$ induced by these two triples are the same. \qed

We now assume that $[\phi]$ satisfies the conditions 1-3 of Section \ref{sec:parms}.

\begin{pro} \label{pro:basepoint} Let $o,o' \in \mc{B}^\tx{red}(G,F)$ be hyperspecial, and choose
$(S_0,[a],[^Lj]) \in \Xi([\phi],o)$ and $(S_0',[a'],[^Lj']) \in \Xi([\phi],o')$. If  $f : S_{[\phi]}^D \rw \Pi_{[\phi]}$ and $f' : S_{[\phi]}^D \rw
\Pi_{[\phi]}'$ are the corresponding versions of \eqref{eq:packs}, then $\Pi_{[\phi]} = \Pi_{[\phi]}'$ and we have a commutative diagram
\begin{diagram}[small,midshaft]
S_{[\phi]}^D&\rTo^f&\Pi_{[\phi]}\\
\dTo<{+(o,o')}&&\dEquals\\
S_{[\phi]}^D&\rTo^{f'}&\Pi_{[\phi]}'
\end{diagram}
In particular, the bijection $f$ depends only on the $G(F)$-orbit of $o$.
\end{pro}
\pf Fix representatives $a,{^Lj}$ in the classes $[a],[^Lj]$ and put $\phi = {^Lj}\circ a$. As in step 2, $^Lj$ gives a map $S_\phi^D \rw X_*(S_0)_\Gamma$
which we will denote by $^Lj^D$. Let $\rho \in S_\phi^D$ and $\lambda \in X_*(S_0)$ be such that the image of $\rho$ under $^Lj^D$ coincides with the class of
$\lambda$ in $X_*(S_0)_\Gamma$.

Let $g \in G(\ol{F})$ be such that its image in $G_\tx{ad}$ is defined over $F$ and maps $o$ to $o'$. By Lemma \ref{lem:tripuni} we may choose $g$ so that
$(S_0',[a'],[^Lj']) = \tx{Ad}(g)(S_0,[a],[^Lj])$. Let $\lambda_g \in X_*(S_0)$ be an element whose image in $X_*(S_0)_\Gamma$ corresponds to $(o,o')$ under
$^Lj^D$. Then $\lambda' := \tx{Ad}(g)[\lambda+\lambda_g]$ is an element of $X_*(S_0')$, $^Lj' := {^Lj}\circ\hat{\tx{Ad}}(g)^{-1}$ is a representative of the
class $[^Lj']$, and the image of $\rho+(o,o')$ under $^Lj'^D$ coincides with the class of $\lambda'$ in $X_*(S_0')_\Gamma$.

As in step 2 of the previous section, we obtain from $(S_0,\theta_0,\lambda)$ a quadruple $(G^\lambda,\psi_\lambda,b_\lambda,\pi_\lambda)$. This quadruple is
the image of $\rho$ under $f$. In the same way, we obtain from $(S_0',\theta_0',\lambda')$ a quadruple
$(G^{\lambda'},\psi_{\lambda'},b_{\lambda'},\pi_{\lambda'})$, which is the image of $\rho+(o,o')$ under $f'$. We claim that the two quadruples are equivalent.
Indeed, there exists an $s \in S_0(\ol{F})$ with $gs \in G(F^u)$ and $(gs)^{-1}\Phi(gs)=\lambda_g(\pi)^{-1}$. A straightforward computation shows that
\[ (\psi_{\lambda'},b_{\lambda'}) \circ(\tx{Ad}(g),(gs)^{-1}\Phi(gs))\circ (\psi_\lambda,b_\lambda)^{-1} \]
is a strongly-trivial twist which identifies $(S_\lambda,\theta_\lambda)$ with $(S_{\lambda'},\theta_{\lambda'})$ and hence $\pi_\lambda$ with
$\pi_{\lambda'}$.\qed

\subsection{Relation to the construction of DeBacker and Reeder} \label{sec:oldpacks}

We continue with $G$, $^LG$, and $[\phi]$ as in \ref{sec:packs}. In \cite{DR09} the authors choose a certain representative $\phi$ of $[\phi]$
and construct an $L$-packet $\Pi'_{\phi}$ on all pure inner forms of $G$ as well as a bijection
\[ \pi_0(S_{\phi})^D \rw \Pi'_{\phi}. \]
The goal of this section is to prove the following
\begin{thm} \label{thm:compat}
The set $\Pi'_{\phi}$ can be naturally identified with a subset of $\Pi_{[\phi]}$, and we have a commutative diagram
\begin{diagram}[small]
\pi_0(S_{\phi})^D&\rTo&\Pi'_{\phi}\\
\dInto&&\dInto\\
S_{[\phi]}^D&\rTo&\Pi_{[\phi]}
\end{diagram}
\end{thm}
Let us very briefly recall their construction.

First, the authors choose a quasi-split maximal torus $T \subset G$, a hyperspecial vertex $o \in \mc{A}^\tx{red}(T,F)$, and a $\Gamma$-invariant maximal torus $\hat
T \subset \hat G$ dual to $T$. They write $^LG = \hat G \rtimes W_F$ and choose $\phi$ within its equivalence class so that $\phi_0(I_F) \subset \hat T$, where
$\phi_0$ is the projection of $\phi$ to $\hat G$. Then
\begin{diagram} w: W_F&\rTo^{\phi_0}&N(\hat T,\hat G)&\rTo&\Omega(\hat T,\hat G)&\rTo^{\cong}&\Omega(T,G) \end{diagram}
is an unramified cocycle. They choose an element $\dot w$ of finite order inside $N(T,G)(O_{F^u})$ which lifts $w(\Phi)$.

Let $T^w$ be the twist of $T$ by $w$. As an $L$-group for $T^w$ we can take
\[ ^LT^w = \hat T \rtimes_\phi W_F, \]
where $W_F$ acts on $\hat T$ via its image under $\phi$. In \cite[4.3]{DR09}, a Langlands parameter $\phi_T : W_F \rw {^LT^w}$ is constructed, by letting
$\phi_T|_{I_F} = \phi|_{I_F}$, and letting $\phi_T(\Phi)=\tau \rtimes \Phi$ where $\tau \in \hat T$ is any element whose image under
\[ \hat T \rw \hat G \rw [\hat G]_\tx{ab} \]
belongs to the $\Phi$-twisted conjugacy class of $\phi(\Phi)$. The Langlands parameter $\phi_T$ corresponds to a character $\theta : T^w(F) \rw \C^\times$.

Next, for every $\mu \in X_*(T^w)$ whose image in $X_*(T^w)_\Gamma$ is torsion, DeBacker and Reeder use in \cite[2,4.4]{DR09} the combinatorics of the Bruhat-Tits
building of $G$ to construct an unramified cocycle $u_{\mu} \in Z^1(\Gamma,G)$. Let $(\psi'_{\mu},u_{\mu}) : G \rw G^{u_{\mu}}$ be the corresponding pure inner
twist. Furthermore, they construct $p_{\mu} \in G(F^u)$ such that both the maximal torus $T_{\mu} := \psi'_{\mu}\tx{Ad}(p_{\mu})T$ and the isomorphism
$\psi'_{\mu}\tx{Ad}(p_{\mu}) : T^w \rw T_{\mu}$ are defined over $F$. Via this isomorphism, the character $\theta$ can be transported to a character
$\theta_{\mu} : T_{\mu}(F) \rw \C^\times$, and one obtains a representation $\pi'_{\mu} := \pi_{G^{u_{\mu}},T_{\mu},\chi_{\mu}}$ as in Section \ref{sec:reps}.
It is argued in \cite[4.5.3]{DR09} and \cite[6.1]{Ree08} that the equivalence class of the quadruple $(G^{u_{\mu}},\psi'_{\mu},u_{\mu},\pi'_{\mu})$ depends
only on the image of ${\mu}$ in $X_*(T^w)_\Gamma$. The set $\Pi'_{[\phi]}$ is then the set of equivalence classes of such quadruples, and we obtain the
bijection
\[ \pi_0(S_\phi)^D = \pi_0([\hat T^w]^\Gamma)^D = X_*(T^w)_\Gamma \rw \Pi'_{[\phi]}. \]
We will not review the details of the construction. The crucial property that will be important for us is  \cite[2.7.(8)]{DR09}, namely
\[ p_{\mu}^{-1}u_{\mu}\Phi(p_{\mu}) = {\mu}(\pi)\dot w. \]
Moreover, $u_0=1$, and $p_0 \in G(O_{F^u})$.

\begin{lem} There exists an $L$-embedding $^Lj_0 : {^LT_0} \rw {^LG}$ extending the admissible isomorphism $\hat{\tx{Ad}}(p_0) : \hat T_0 \rw \hat T$, such
that the triple $(T_0,[a_0],[^Lj_0])$ with $a_0 = \hat{\tx{Ad}}(p_0)^{-1}\circ\phi_T$ belongs to $\Xi([\phi],o)$.
\end{lem}
\pf The extension of $\hat{\tx{Ad}}(p_0)$ to an $L$-embedding $^Lj_0$ is provided by the construction \cite[2.6]{LS87} after choosing unramified $\chi$-data.
The torus $T_0$ is conjugate to $T$ by $p_0$, hence defined over $O_{F^u}$. We only need to check that $[^Lj_0\circ a_0]=[\phi]$. It is clear that this
equality holds when both sides are restricted to $I_F$. To show that both sides are also equal when evaluated at $\Phi$, we argue as follows. Let $a_1 : W_F
\rw {^LT_0}$ be a homomorphism such that $[^Lj_0 \circ a_1]=[\phi]$ (see the argument in Step 1 of \ref{sec:packs}). The Langlands parameters $a_0$ and $a_1$
coincide on inertia, hence the difference of characters on $T_0(F)$ that they correspond to is unramified. To show that this difference is in fact trivial, and
hence $[a_0]=[a_1]$, we argue again as in the proof of Lemma \ref{lem:tripuni}. By \cite[7.1.1]{Kal09} it is enough to show that the difference is trivial when
restricted to $Z(G)^\circ(F) \subset S_0(F)$. This restriction of the character corresponding to $a_0$ has the Langlands parameter
\[ ^Lp \circ [a_0] = {^Lq}\circ[{^Lj_0}\circ a_0] = {^Lq}\circ\phi\]
by construction of $\phi_T$. The restriction of the character corresponding to $a_1$ has the Langlands parameter
\[ ^Lp \circ [a_1] = {^Lq}\circ[{^Lj_0}\circ a_1] = {^Lq}\circ\phi\]
by construction of $a_1$.\qed

\pf (of Theorem \ref{thm:compat})\\
We keep the choices of $o$,$T$,$\hat T$,$\phi$ made so far. Let $\rho \in \pi_0(S_\phi)^D$, and let $\mu \in X_*(T^w)$ be an element whose image in
$X_*(T^w)_\Gamma$ equals the image of $\rho$ under the identification $\pi_0(S_\phi)^D=[X_*(T^w)_\Gamma]_\tx{tor}$. Let $\lambda=\tx{Ad}(p_0)\mu \in X_*(T_0)$.
By the preceding Lemma, we can use the triple $(T_0,[a_0],[^Lj_0])$ for our construction of \eqref{eq:packs}. Then we see that the image of $\rho$ under the
map
\[\pi_0(S_\phi)^D \rw S_\phi^D \rw X_*(T_0)_\Gamma\]
given by $^Lj_0$ equals the class of $\lambda$. Thus we need to show that the quadruple $(G^{u_{\mu}},\psi'_{\mu},u_{\mu},\pi'_{\mu})$ constructed by
DeBacker-Reeder resp. Reeder and the quadruple $(G^\lambda,\psi_\lambda,b_\lambda,\pi_\lambda)$ constructed in Section \ref{sec:packs} are equivalent. By Lemma
\ref{lem:pairiso} we need to find a strongly trivial twist
\[(G^\lambda,\psi_\lambda,b_\lambda) \rw (G^{u_{\mu}},\psi'_{\mu},u_{\mu}), \]
which carries the pair $(S_\lambda,\theta_\lambda)$ to the pair $(T_{\mu},\theta_{\mu})$. We claim that
\[ (\psi'_\mu,u_\mu)\circ\tx{Ad}(p_\mu p_0^{-1})\circ(\psi_\lambda,b_\lambda)^{-1} \]
is a strongly-trivial twist. Computing the cocycle of this twist we obtain
\begin{eqnarray*}
&&\psi_\lambda(\tx{Ad}[p_\mu p_0^{-1}]^{-1}(u_\mu)(p_\mu p_0^{-1})^{-1}\Phi(p_\mu p_0^{-1})b_\lambda^{-1})\\
&=&\psi_\lambda(p_0 p_\mu^{-1}u_\mu\Phi(p_\mu)\Phi(p_0^{-1})b_\lambda^{-1})\\
&=&\psi_\lambda(p_0\mu(\pi)\dot w\Phi(p_0^{-1})b_\lambda)\\
&=&\psi_\lambda(\lambda(\pi)\tx{Ad}[p_0](\dot w \Phi(p_0^{-1})p_0)b_\lambda^{-1})\\
&=&1.
\end{eqnarray*}
If $\theta_0 : T_0(F) \rw \C^\times$ is the character corresponding to the parameter $a_0$, then we have $\theta_0 = \tx{Ad}(p_0)\theta$, and hence the above
twists carries the pair $(S_\lambda,\theta_\lambda)$ to the pair $(T_\mu,\theta_\mu)$.\qed

\subsection{Generic representations} \label{sec:generic}

We saw in Proposition \ref{pro:basepoint} that the bijection \eqref{eq:packs} depends only the the $G(F)$-orbit of the chosen hyperspecial vertex $o$. The
image of the trivial character under this bijection is a generic representation, and the set of generic characters with respect to which this representation is
generic can also be read off from the vertex $o$.

For a hyperspecial vertex $o'$ and $r \geq 0$ consider the set of pairs
\[\tx{Gen}[o',r] = \{ (B,\psi) \}, \]
where $B$ is a Borel subgroup of $G$ such that the reduced apartment of some maximal torus of $B$ contains $o'$, and $\psi : B_u(F) \rw \C^\times$
is a character of generic depth $r$ at $o'$. Let
\[ \tx{Gen}[G(F)o,r] = \bigcup_{o' \in G(F)o} \tx{Gen}[o',r]. \]

\begin{fct} Let $o$ a hyperspecial vertex, and $\pi$ the image of the trivial character on $S_{[\phi]}$ under the version of
\eqref{eq:packs} corresponding to $o$. Then for each $(B,\psi) \in \tx{Gen}[G(F)o,0]$ we have $\tx{Hom}_{B_u(F)}(\pi,\psi)\neq 0$.
\end{fct}
\pf This is the content of \cite[\S6]{DR09}.\qed

Recall that an $F$-splitting of $G$ is a triple $(T,B,\{X_\alpha\})$ where $(T,B)$ is a Borel pair, and $X_\alpha \in \mf{g}_\alpha(\ol{F})$ is a non-trivial
element for each simple root of $T$ in $B$, and the set $\{X_\alpha\}$ is $\Gamma$-invariant. Any pair $(T,B)$ can be extended to an $F$-splitting. Moreover,
the isomorphism $\mb{G}_a \rw \mf{g}_\alpha$ given by $1 \mapsto X_\alpha$ exponentiates to a homomorphism $x_\alpha : \mb{G}_a \rw G$ whose image is the
unique 1-dimensional subgroup of $G$ which $T$ normalizes and acts on by the character $\alpha$. If $G$ is defined over $O_F$, an $F$-splitting of $G$ will be
called an $O_F$-splitting if $T$ is defined over $O_F$, and each homomorphism $x_\alpha$ is defined over $O_F$ and induces a non-trivial homomorphism on the
special fibers.

Let
\[ \tx{Spl}[o'] = \{(T,B,\{X_\alpha\})\} \]
consist of the $O_F$-splittings of $G$ for the $O_F$-structure defined by $o'$, and let
\[ \tx{Spl}[G(F)o] = \bigcup_{o' \in G(F)o} \tx{Spl}[o']. \]
Fix a character $\Lambda : F \rw \C^\times$ which is trivial on $\pi^{r+1} O_F$ but non-trivial on $\pi^r O_F$. An element $(T,B,\{X_\alpha\}) \in
\tx{Spl}[o']$ gives rise to an element $(B,\psi) \in \tx{Gen}[o',r]$ as follows. The map
\[ \prod_{\alpha \in \Delta(T,B)} \mb{G}_a \rw B_u \rw B_u^\tx{ab},\qquad (\xi_\alpha) \mapsto \prod x_\alpha(\xi_\alpha) \]
is an isomorphism of $\ol{F}$-groups. Composing its inverse with the summation map gives a homomorphism of $F$-groups $B_u^\tx{ab} \rw \mb{G}_a$, which then
composed with $\Lambda$ provides a generic depth $r$ character $\psi : B_u(F) \rw \C^\times$.

\begin{pro} \label{pro:genchar} The map $(T,B,\{X_\alpha\}) \mapsto (B,\psi)$ is $G_\tx{ad}(F)$-equivariant. It induces a surjection
\[ \tx{Spl}[G(F)o] \rw \tx{Gen}[G(F)o,r] \]
whose fiber over a given $(B,\psi)$ is a $B_u(F)$-orbit. In particular, it induces a bijection
\[ \tx{Spl}[G(F)o]/\tx{Ad}(G(F)) \rw \tx{Gen}[G(F)o,r]/\tx{Ad}(G(F)). \]
\end{pro}
\pf The $G_\tx{ad}(F)$-equivariance is obvious from the construction. To show surjectivity, fix $(B,\psi) \in \tx{Gen}[o,r]$. Choose any $F$-splitting $(T',B',\{X_\alpha'\})$ and let $(B',\psi')$ be the corresponding generic character. Since $(B,\psi)$ and $(B',\psi')$ are conjugate under $G_\tx{ad}(F)$, we see that there exists a splitting $(T,B,\{X_\alpha\})$ mapping to $(B,\psi)$. Conjugating this splitting by $B_u(F)$ we may assume that the reduced apartment of $T$ contains $o$. The fact that the depths of $\psi$ and $\Lambda$ are equal forces $(T,B,\{X_\alpha\}) \in \tx{Spl}[o]$.

Now consider two splittings $(T,B,\{X_\alpha\})$ and $(T',B,\{X'_\alpha\})$ in the fiber over $(B,\psi)$. Conjugating by $B_u(F)$ we may assume that $T=T'$. There exists a unique element $t \in T_\tx{ad}(F)$ sending $\{X_\alpha\}$ to $\{X'_\alpha\}$. We wish to show that in fact $t=1$. By assumption, $\tx{Ad}(t)$ preserves $\psi$. Let $\alpha \in
\Delta(T,B)$ and let $E/F$ be the field of definition of $\alpha$. If $U_{(\alpha)}$ is the relative root subgroup of $G$ defined by the restriction
$(\alpha)$ of $\alpha$ to the split part of $T$, then
\[ E \rw [U_{(\alpha)}]_\tx{ab}(F),\quad \xi \mapsto \prod_{\sigma\in\tx{Gal}(E/F)} x_{\sigma\alpha}(\sigma\xi) \]
is an isomorphism of groups. Composing $\psi$ with this isomorphism yields a non-trivial character on $E$. The automorphism $\tx{Ad}(t)$ pre-composes this character with multiplication by $\alpha(t) \in E^\times$. Thus, the assumption that $\tx{Ad}(t)$ preserves $\psi$ implies that $\alpha(t)=1$ for all $\alpha \in \Delta(T,B)$, hence $t=1$. \qed

\section{Stability and endoscopic transfer} \label{sec:stabendo}

We keep the notation established so far. In particular, $G$ is an unramified connected reductive group defined over $F$, $^LG$ is an $L$-group for $G$, and $[\phi]$ is an equivalence class of
Langlands parameters satisfying conditions 1-3 of Section \ref{sec:parms}.

\subsection{Preparatory lemmas} \label{sec:seprep}

Let $J$ be a connected reductive group defined over $F$ with Lie algebra $\mf{j}$, $B$ a conjugation\-invariant bilinear form on $\mf{j}(F)$, and $\chi : F \rw \C^\times$ an
additive character. To this data, one can define the Weil constant $\gamma_\chi(B)$, see \cite[VIII]{Wal95}. If $J'$ is an inner form, or an endoscopic group
of $J$, then Waldspurger shows in loc. cit. how to transfer $B$ to a conjutation\-invariant bilinear form $B'$ on $\mf{j'}(F)$, and the ratio
$\gamma_\chi(B)\gamma_\chi(B')^{-1}$ plays an important role in Waldspurger's work on endoscopy for $p$-adic Lie algebras. In order to prove the stability and
endoscopic transfer of our depth-zero packets, we need to know that when $J$ and $J'$ split over $F^u$, this ratio can be expressed in terms of the $F$-split ranks of
$J$ and $J'$. This has been worked out when $J$ and $J'$ are pure inner forms in \cite[\S12]{DR09}, and when $J$ is quasi-split and $J'$ is an endoscopic group
in \cite[\S4]{Kal09}. Here we build on the arguments of \cite[\S12]{DR09} to handle the case when $J$ and $J'$ are general inner forms. Of course it is enough
to assume that $J$ is quasi-split.

\begin{lem} \label{lem:sign} Let $J$ be an unramified connected reductive group defined over $F$ and $\psi : J \rw J'$ an inner twist. Let $B$ be an $\tx{Ad}(J(F))$-invariant bilinear form on $\tx{Lie}(J)(F)$,
and $\chi : F \rw \C^\times$ an additive character. Then
\[ e(J') = \gamma_\chi(B)\gamma_\chi(B')^{-1}, \]
where $B'$ is the corresponding bilinear form on $\tx{Lie}(J')(F)$, and $e$ is the Kottwitz\-sign.
\end{lem}
\pf Let $T' \subset J'$ be an unramified elliptic maximal torus, which exists by \cite[2.4]{Deb06}. Up to equivalence of $\psi$ we may assume that the maximal torus $T:=\psi^{-1}(T')$ is defined
over $F$ and $\psi : T \rw T'$ is an $F$-isomorphism. Recall that an orbit of $\Gamma$ in $R(T,J)$ is called symmetric if it is invariant under multiplication
by $-1$. Let $\tx{Sym}(T)$ be a set of representatives for the symmetric orbits. For each $\alpha \in \tx{Sym}(T)$, let $F_\alpha$ be the fixed field of the
stabilizer of $\alpha$ in $\Gamma$, and $F_{\pm\alpha}$ be the fixed field of the stabilizer of $\{\alpha,-\alpha\}$ in $\Gamma$. The root subspace
$\mf{j}_\alpha \subset \mf{j}$ is defined over $F_\alpha$ and we may choose $E_\alpha \in \mf{j}_\alpha(F_\alpha)$. It was shown by Waldspurger
\cite[VIII.5]{Wal95} that
\[ \gamma_\chi(B) = \gamma_\chi(B|_{\mf{t}(F)}) \prod_{\alpha \in\tx{Sym}(T)} \gamma_{\chi_\alpha}(q_\alpha), \]
where $\chi_\alpha : F_{\pm\alpha} \rw \C^\times$ is the character $\chi\circ\tr_{F_{\pm\alpha}/F}$, and $q_\alpha : F_\alpha \rw F_{\pm\alpha}$ is the
quadratic form $\lambda \mapsto \tx{Nm}_{F_\alpha/F_{\pm\alpha}}(\lambda)B(E_\alpha,E_{-\alpha})$. Let $d$ be the level of $\chi$ (the largest integer such
that $\chi$ is trivial on the ideal $\pi^{-d}O_F$). Since $\tr_{F_{\pm\alpha}/F} : O_{F_{\pm\alpha}} \rw O_F$ is surjective, $d$ is also the level of
$\chi_\alpha$. It is known (\cite[Proof of Lemma 1.2]{JL70}) that
\[ \gamma_{\chi_\alpha}(q_\alpha) = (-1)^{d+\tx{val}(B(E_\alpha,E_{-\alpha}))}. \]
A corresponding formula holds for $\gamma_\psi(B')$ as well and since $\psi$ provides an isomorphism $(\mf{t}(F),B) \rw (\mf{t}'(F),B')$ and a bijection
$\tx{Sym}(T) \rw \tx{Sym}(T')$ we obtain
\[ \gamma_\chi(B)\gamma_\chi(B')^{-1} = \prod_{\alpha\in\tx{Sym}(T)}(-1)^{\tx{val}\{B(E_\alpha,E_{-\alpha)})B'(E'_{\psi(\alpha)},E'_{-\psi(\alpha)})\}}. \]
There exists $c_\alpha \in F_\alpha^\times$ such that $E'_{\psi(\alpha)} = c_\alpha\psi(E_\alpha)$ and the right hand side becomes
$\prod(-1)^{\tx{val}(c_\alpha c_{-\alpha})}$.

Let $a_\sigma \in Z^1(\Gamma,J_\tx{ad})$ be such that $\psi^{-1}\sigma(\psi) = \tx{Ad}(a_\sigma)$. By our assumption $a_\sigma \in Z^1(\Gamma,T_\tx{ad})$.
There exists $\lambda \in X_*(T_\tx{ad})$ with $a_\Phi=\lambda(\pi)$ \cite[2.3.3]{Kal09}. The computations of \cite[12.3.2+12.3.3]{DR09} show that
\[ \prod_{\alpha\in\tx{Sym}(T)} (-1)^{\tx{val}(c_\alpha c_{-\alpha})} = (-1)^{\langle\lambda,2\rho\rangle}. \]
where $2\rho$ is the sum of any choice of positive roots in $R(T,J)$.

By definition, $e(J')$ is the image of $a_\sigma$ under
\[ H^1(F,T_\tx{ad}) \rw H^1(F,J_\tx{ad}) \stackrel{\partial}{\lrw} H^2(F,Z(J_\tx{sc})) \stackrel{\rho}{\lrw} H^2(F,\mb{G}_m) \stackrel{\exp(2\pi i\cdot\tx{inv})}{\lrw} \C^\times. \]
Tate-Nakayama-duality identifies the profinite completion of $H^0(F,X^*(\mb{G}_m))$ with the character group of $H^2(F,\mb{G}_m)$. One has
\[ \exp(2\pi i \cdot \tx{inv}(x)) = 1(x),\quad\tx{for}\ x \in H^2(F,\mb{G}_m), 1 \in \Z \stackrel{\cong}{\longrightarrow} H^0(F,X^*(\mb{G}_m)).\]
Hence, applying Tate-Nakayama-duality to the above sequence of arrows we see that $e(J')$ is equal to the pairing of $a$ with the image of $1$ under the dual
sequence
\[ H^0(F,X^*(\mb{G}_m)) \rw H^0(F,X^*(Z(J_\tx{sc}))) \rw H^1(F,X^*(T_\tx{ad})) \]
This image is represented by the cocycle $\Phi \mapsto \Phi(\rho)-\rho$. Since $a$ is the image of $-\lambda$ under the Tate-Nakayama isomorphism
$H^{-1}_\tx{Tate}(F,X_*(T_\tx{ad})) \rw H^1(F,T_\tx{ad})$ one sees by applying \cite[2.3.2]{Kal09} that the pairing of $a$ and $\Phi\mapsto \Phi(\rho)-\rho$ is
$\exp(2\pi i\langle\lambda,\rho\rangle)$.\qed

\subsection{The unstable character} \label{sec:schar}

Let $\phi : W_F \rw {^LG}$ be a Langlands parameter as in Section \ref{sec:parms} whose depth is zero. We fix a hyperspecial vertex $o \in
\mc{B}^\tx{red}(G,F)$ and obtain the corresponding version of the bijection \eqref{eq:packs}. To each $\rho \in S_{[\phi]}^D$, this bijection assigns an
equivalence class of quadruples $\{ (G_b,\psi,b,\pi) \}$. Let $\rho,\rho' \in S_{[\phi]}^D$ and $(G_b,\psi,b,\pi)$ and $(G_{b'},\psi',b',\pi')$ be representatives
within the corresponding equivalence classes. We will call an ep twist $G_b \rw G_{b'}$ \emph{admissible} if it is equivalent to $\psi'\circ\psi^{-1}$. If $\rho,\rho'$ have the same image $\bar\rho$ in $[Z(\hat G)^\Gamma]^D$, then there exists an admissible
strongly trivial ep twists $G_b \rw G_{b'}$, and every such induces the same bijection $\tx{Cl}(G_{b,\tx{sr}}(F)) \rw \tx{Cl}(G_{b',\tx{sr}}(F))$ between the conjugacy
classes of strongly regular semi-simple elements. Hence we may speak of the set $\tx{Cl}_{\bar\rho}$, and the characters of $\pi$ and $\pi'$ provide functions
$\Theta_\rho$ and $\Theta_{\rho'}$ on $\tx{Cl}_{\bar\rho}$. These functions are equal if and only if $\rho=\rho'$. Moreover, the Kottwitz signs of $G_b$ and $G_{b'}$
are the same and we can write $e_{\bar\rho}$ for them.

\begin{dfn}
Let $t \in S_{[\phi]}$. We construct a function $\Theta_{[\phi]}^t$ on $\bigsqcup\limits_{\xi\in [Z(\hat G)^\Gamma]^D} \tx{Cl}_\xi$, which is given on each $\tx{Cl}_\xi$
by
\[ \Theta_{[\phi]}^t = e_{\xi}\sum_{\rho} \rho(t)\Theta_{\rho},\]
where $\rho$ runs over the fiber of $S_{[\phi]}^D \rw [Z(\hat G)^\Gamma]^D$ over $\xi$.
\end{dfn}

Fix an element $(S_0,[a],[^Lj]) \in \Xi([\phi],o)$. We obtain a map
\begin{equation} \label{eq:x} X_*(S_0)_\Gamma \stackrel{\cong}{\lrw} S_{[\phi]}^D. \end{equation}
For $\lambda \in X_*(S_0)$, let $\rho_\lambda \in S_{[\phi]}^D$ be its image. Then the quadruple $(G^\lambda,\psi_\lambda,b_\lambda,\pi_\lambda)$ constructed
in Section \ref{sec:packs} is a representative of the element of $\Pi_{[\phi]}$ corresponding to $\rho_\lambda$ under \eqref{eq:packs}.

\begin{pro} \label{pro:schar} Let $Q_0 \in\tx{Lie}(S_0)(F)$ be a regular semi-simple element. For any $\gamma \in G^\lambda_\tx{sr}(F)_0$ and $z \in Z(F)$, the value of $\Theta_{[\phi]}^t(z\gamma)$ is given by
\[ \epsilon(G,A_G)\theta_0(z)\sum_P [\phi_{Q_0,P}]_*\theta_0(\gamma_s)
\sum_Q \langle \tx{inv}(Q_0,Q),t\rangle^{-1}R(G^\lambda_{\gamma_s},S_Q,1)(\gamma_u) \]
where $P$ runs over a set of representatives for the $G^\lambda_{\gamma_s}$-stable classes of elements of $\tx{Lie}(G^\lambda_{\gamma_s})(F)$ which are
$G^\lambda$-stably conjugate to $Q_0$, and $Q$ runs over a set of representatives for the $G^\lambda_{\gamma_s}(F)$-conjugacy classes inside the
$G^\lambda_{\gamma_s}$-stable class of $P$, and $\langle\tx{inv}(Q_0,Q),-\rangle$ is the image of $\tx{inv}(Q_0,Q)$ under the map $\tb{B}(S_0) \rw
S_{[\phi]}^D$.
\end{pro}
\pf Let $M \subset X_*(S_0)$ be a set of representatives for the fiber of
\begin{diagram}
X_*(S_0)_\Gamma&\rTo^{\footnotesize\eqref{eq:x}}&S_{[\phi]}^D&\rTo&[Z(\hat G)^\Gamma]^D
\end{diagram}
through $\lambda$. For each $\mu\in M$, and let $\psi_{\mu,\lambda} : G^\mu \rw G^\lambda$ be a strongly trivial admissible ep twist. Put $(S_\mu',\theta_\mu')
= \psi_{\mu,\lambda}(S_\mu,\theta_\mu)$. Then $\Theta_{\rho_\mu}$ equals the character of $\pi_{G^\lambda,S_\mu',\theta_\mu'}$. Hence using Lemma
\ref{lem:1char} and noting that $\theta_0$ and $\theta_\mu'$ coincide on $Z(F)$ we obtain
\[ \Theta_{[\phi]}^t(z\gamma) = \epsilon(G,A_G)\theta_0(z) \sum_{\mu\in M} \rho_{\mu}(t)\sum_{Q} R(G^\lambda_{\gamma_s},S_Q,1)(\gamma_u)[\phi_{Q_\mu,Q}]_*\theta_\mu'(\gamma_s), \]
where $Q_\mu = \psi_{\mu,\lambda}\psi_\mu(Q_0)$ and $Q$ runs over any set of representatives for the $G^\lambda_{\gamma_s}(F)$-conjugacy classes inside the
$G^\lambda(F)$-conjugacy class of $Q_\mu$. We have
\[ \rho_{\mu}(t) = \langle b_\mu^{-1},t \rangle = \langle\tx{inv}(Q_0,Q_\mu),t\rangle^{-1} = \langle\tx{inv}(Q_0,Q),t\rangle^{-1} \]
The reason for the $-$-sign is that $b_\mu$ is the image of $-\mu$ under the isomorphism $X_*(S_0)_\Gamma \rw \tb{B}(S_0)$, and the last equality holds because
$Q_\mu$ and $Q$ are $G^\lambda(F)$-conjugate. Recalling $\theta_\mu'=[\psi_{\lambda,\mu}\psi_\mu]_*\theta_0$ we arrive at
\[ \Theta_{[\phi]}^t(z\gamma) = \epsilon(G,A_G)\theta_0(z)\!\!\sum_{\mu}\!\!\sum_{Q}\!\langle\tx{inv}(Q_0,Q),t\rangle^{-1}\!R(G^\lambda_{\gamma_s},S_Q,1)(\gamma_u)[\phi_{Q_0,Q}]_*\theta_0(\gamma_s) \]
Applying the Lie algebra version of \cite[2.1.5]{Kal09} in our context we see that as $\mu$ runs over the fiber of $X_*(S_0)_\Gamma \rw \tb{B}(G)_b$ through
$\lambda$, $Q_\mu$ runs over a set of representatives for the $G^\lambda(F)$-rational classes inside the intersection of the stable class of $Q_0$ with
$\tx{Lie}(G^\lambda)(F)$. Thus in the double sum, $Q$ runs over a set of representatives for the $G^\lambda_{\gamma_s}(F)$-classes inside the intersection of
the stable class of $Q_0$ with $\tx{Lie}(G^\lambda)(F)$. We can rewrite this sum as
\[ \Theta_{[\phi]}^t(z\gamma) = \epsilon(G,A_G)\theta_0(z)\!\!\sum_{P}\!\!\sum_{Q}\!\langle\tx{inv}(Q_0,Q),t\rangle^{-1}\!R(G^\lambda_{\gamma_s},S_Q,1)(\gamma_u)[\phi_{Q_0,Q}]_*\theta_0(\gamma_s) \]
where $P$ and $Q$ run as in the statement of the lemma. To conclude the proof, we only need to notice that if $Q,Q'$ are $G^\lambda_{\gamma_s}$-conjugate, then
$[\phi_{Q_0,Q}]_*\theta_0(\gamma_s)=[\phi_{Q_0,Q'}]_*\theta_0(\gamma_s)$.\qed

\begin{lem} \label{lem:ha} Let $\gamma^\lambda \in G^\lambda(F)_0$ be strongly regular semi-simple and assume that $\gamma^\lambda_s \in S(F)$ for some maximal torus $S \subset G^\lambda$
which is stably conjugate to $S_0$. Then there exists $\gamma \in G(F)$ stably conjugate to $\gamma^\lambda$ with $\gamma_s \in S_0(F)$.
\end{lem}
\pf The argument is the same as for \cite[7.2.2]{Kal09}, but instead of using $\tx{Ad}(q_0q_\lambda^{-1})$ we use any admissible ep twist sending $S$ to
$S_0$.\qed

\subsection{Stability and endoscopy} \label{sec:se}

In this section we need to impose further restrictions on $F$. Let $n_G$ be the smallest dimension of a faithful representation of $G$, $e$ the ramification
degree of $F/\Q_p$, $e_G$ the maximum ramification degree over $\Q_p$ of a splitting field of a maximal torus of $G$, and $\nu(G)$ be the number of positive
roots of $G$. For Theorem \ref{thm:stab}, we require $q_F \geq \nu(G)$ and $p \geq (2+e)n_G$. For Theorem \ref{thm:endo} we require in addition $p \geq 2+e_G$
and $p \geq (2+e)n_H$ ($H$ the corresponding endoscopic group).

\begin{thm} \label{thm:stab} The function $\Theta_{[\phi]}^1$ is stable. That is, if $\gamma^\lambda \in G^\lambda_\tx{sr}(F),\gamma^\mu \in G^\mu_\tx{sr}(F)$ are stably conjugate,
where $\lambda,\mu \in X_*(S_0)$, then $\Theta_{[\phi]}^1(\gamma^\lambda)=\Theta_{[\phi]}^1(\gamma^\mu)$.
\end{thm}
\pf We have $\gamma^\lambda \in Z(F)G^\lambda_\tx{sr}(F)_0$ if and only if $\gamma^\mu \in Z(F)G^\mu_\tx{sr}(F)_0$. We may assume this is the case, for
otherwise both sides of the equation vanish \cite[9.6.3]{DR09}. Write $\gamma^\lambda=zsu$ with $z\in Z(F)$, $s$ top. semi-simple and $u$ top. unipotent. By
Lemma \ref{lem:ha} there exists $\gamma \in G(F)$ stably conjugate to $\gamma^\lambda$, with corresponding decomposition $\gamma=ztv$ and $t \in S_0(F)$. It
will be enough to show the proposition in the case $\mu=0$ and $\gamma^\mu=\gamma$.

Choose a semi-simple element $Q_0 \in \tx{Lie}(S_0)(O_F)$ with strongly-regular reduction; such an element exists, see \cite[12.4.2]{DR09}. Let $(\psi,b) :
G^\lambda \rw G$ be an admissible ep twist sending $\gamma^\lambda$ to $\gamma$. The ep twist $(\psi,b)$ restricts to an ep twist
\[ (\psi,b) : G^\lambda_{s} \rw G_{t} \]
and provides an injection of stable classes, both on the level of the groups and their Lie algebras. We will denote this injection again by $\psi$. Since
elliptic tori transfer to inner forms, this injection restricts to a bijection between those stable classes which are stably conjugate to $Q_0$ under
$G^\lambda$ resp. $G$. From Proposition \ref{pro:schar} and \cite[12.4.3]{DR09} we have
\[ \Theta_{[\phi]}^1(\gamma^\lambda)=\epsilon(G,A_G)\epsilon(G^\lambda_s,A_{G^\lambda_{s}})\theta_0(z) \sum_{P_\lambda} [\phi_{Q_0,P_\lambda}]_*\theta_0(s) \hat S_{P_\lambda}^{G^\lambda_s}(\log(u)) \]
\[ \Theta_{[\phi]}^1(\gamma)=\epsilon(G,A_G)\epsilon(G_{t},A_{G_{t}})\theta_0(z) \sum_{P} [\phi_{Q_0,P}]_*\theta_0(t) \hat S_{P}^{G_{t}}(\log(v)) \]
where $P_\lambda$ runs over a set of representatives for the $G^\lambda_{s}$-stable classes of elements of $\mf{g}^\lambda_{s}(F)$ which are stably conjugate
to $Q_0$, $P$ runs over a set of representatives for the $G_{t}$-stable classes of elements of $\mf{g}_{t}(F)$ which are stably conjugate to $Q_0$, and $\hat
S^{G^\lambda_{s}}_{P_\lambda}$ denotes the Fourier transform of the stable orbital integral of $P_\lambda$ in $\mf{g}^\lambda_{s}$. From \cite[Conj 1.2]{Wal97}
and Lemma \ref{lem:sign} we know
\[ \hat S_{P_\lambda}^{G^\lambda_{s}}(\log(u)) = \epsilon(G_{t},G^\lambda_{s}) \hat S_{\psi(P_\lambda)}^{G_{t}}(\log(v)) \]
Since $A_{G^\lambda_{s}}=A_{G_{t}}$, the lemma follows.\qed

Let $(H,s,{^L\eta})$ be an extended unramified endoscopic triple for $G$. Let $[\phi^H] : W_F \rw {^LH}$ be a Langlands parameter and put $[\phi] = [^L\eta]\circ[\phi]$. Assume that $\phi$ is a parameter of the type described in Section \ref{sec:parms}. Then so is $\phi^H$.

Fix a hyperspecial vertex $o \in \mc{B}^\tx{red}(G,F)$. To each element of $\tx{Gen}[G(F)o,0]$, we have the Whittaker
normalization of the transfer factor for $(G,H)$, as explained in \cite[5.3]{KS99}. In fact, the particular choice of element is irrelevant:

\begin{pro} Each element of $\tx{Gen}[G(F)o,0]$ determines the same Whittaker normalization of the transfer factor. \end{pro}
\pf Let $(B,\psi)$ and $(B',\psi')$ belong to $\tx{Gen}[G(F)o,0]$. Conjugating the Whittaker data by $G(F)$ has no effect on the transfer factors, so we may
assume $B=B'$, as well as $(B,\psi),(B,\psi') \in \tx{Gen}[o,0]$. Let $(T,B,\{X_\alpha\})$ and $(T,B,\{X'_\alpha\})$ be two elements of $\tx{Spl}[o,0]$ which give rise to $\psi$ resp. $\psi'$ as in
Proposition \ref{pro:genchar}. The normalization of transfer factor corresponding to the Whittaker data $(B,\psi)$ differs from the normalization corresponding
to the splitting $(T,B,\{X_\alpha\})$ by a term depending on $\Lambda$. Hence it is enough to show that the two splittings $(T,B,\{X_\alpha\})$ and
$(T,B,\{X'_\alpha\})$ lead to the same normalization, which is the content of \cite[7.2]{Hal93}.\qed

Thus the Whittaker normalization of the transfer factor depends only on the $G(F)$-orbit of $o$. We will write $\Delta_o$ for it. Of course, it also depends on
$\Lambda$ and the extended triple $(H,s,{^L\eta})$, but those we assume fixed.

Let $\lambda \in X_*(S_0)$ and $(G^\lambda,\psi_\lambda,b_\lambda)$ be the corresponding ep twist. Let $\Delta_{o,\lambda}$ be the normalization of the
transfer factor for $(G^\lambda,H)$ constructed from $\Delta_o$ in Section \ref{sec:tf}. With respect to this transfer factor we have the endoscopic lift of
the stable character of the $L$-packet $\Pi^H_{[\phi^H]}$
\[ \tx{Lift}^{G^\lambda,o}_H\Theta_{[\phi^H],0}^1(\gamma^\lambda) := \sum_{\gamma^H} \Delta_{o,\lambda}(\gamma^H,\gamma^\lambda)
\frac{D^H(\gamma^H)^2}{D^{G^\lambda}(\gamma^\lambda)^2}\Theta_{[\phi^H],0}^1(\gamma^H), \]%
where $\gamma^H$ runs over the stable conjugacy classes of strongly regular semi\-simple elements of $H(F)$ and $D$ are the usual Weyl discriminants.

\begin{thm} \label{thm:endo} For any  $\gamma^\lambda \in G_\tx{sr}(F)$ we have
\[ \tx{Lift}^{G^\lambda,o}_H\Theta_{[\phi^H],0}^1(\gamma^\lambda) = \Theta_{[\phi]}^s(\gamma^\lambda). \]
\end{thm}
\pf The argument in \cite[\S7]{Kal09} works just as well in this case. The only modifications necessary are that references to Section 2, Proposition 6.2.2,
and Lemma 7.2.2 in loc. cit, and to \cite[\S12.3]{DR09}, have to be replaced with references to Section 2, Proposition \ref{pro:schar}, Lemma \ref{lem:ha} and
Lemma \ref{lem:sign} in this paper.\qed


{\small
Tasho Kaletha\\
tkaletha@math.ias.edu\\
School of Mathematics, Institute for Advanced Study, \\
Einstein Drive, Princeton, NJ 08540
}

\end{document}